\documentclass{amsart}
\usepackage{hyperref}
\hypersetup{
 colorlinks=true,
 linkcolor=blue,
 filecolor=magenta, 
 urlcolor=cyan,
}
\usepackage[margin=1in]{geometry}
\usepackage[style=alphabetic,backend=biber,minalphanames=3]{biblatex}
\usepackage{silence,amsmath,etoolbox,quiver,wasysym,xcolor,xparse,amssymb,amsfonts,stmaryrd,amsthm,mathtools,comment,bm,float}
\newtheorem{theorem}{Theorem}[section]
\newtheorem{lemma}[theorem]{Lemma}
\newtheorem{proposition}[theorem]{Proposition}
\newtheorem{corollary}[theorem]{Corollary}
\newtheorem{computation}[theorem]{Computation}

\newcounter{intro}
\newtheorem{question}{Question}
\newtheorem{introthm}[intro]{Theorem}

\newtheorem{introcomp}[intro]{Computation}

\theoremstyle{definition}
\newtheorem{definition}[theorem]{Definition}
\newtheorem{example}[theorem]{Example}
\newtheorem{remark}[theorem]{Remark}

\newtheorem{construction}[theorem]{Construction}
\newtheorem{notation}[theorem]{Notation}
\usepackage{quiver}
\usepackage[disable]{todonotes}
\usepackage{blkarray}

\allowdisplaybreaks

\newcommand{\twoseg}[2]{\begin{tikzpicture}[baseline=-.7em]\node[circle,fill,scale=.4,label={below: {\scriptsize #1}}] (A) at (0,0) {};\node[circle,fill,scale=.4,label={below: {\scriptsize #2}}] (B) at (.5,0) {};\draw (A) -- (B);\end{tikzpicture}}
\newcommand{\dotseg}[1]{\begin{tikzpicture}[baseline=-.7em]\node[circle,fill,scale=.4,label={below: {\scriptsize #1}}] (A) at (0,0) {};\end{tikzpicture}}
\usepackage{tikz}
\usetikzlibrary{calc,intersections,angles}
\title{Cohomological Support Varieties of Certain Monomial Ideals}
\author{Michael Gintz}
\address{Department of Mathematics, Statistics and Computer Science, The University of Illinois Chicago, 851 S.~Morgan St., Chicago, IL 60607}
\email{mgintz2@uic.edu}
\begin{document}
\begin{abstract}
Building on work of Briggs, Grifo and Pollitz \cite{embdef}, we give an example of two cohomological support varieties of monomial ideals which are not unions of linear subspaces. We provide a procedure for the computation of the cohomological support varieties of certain other monomial ideals—including those which are equigenerated—with improved computational efficiency, leading to a computer-assisted verification of the existence of a third support variety of a monomial ideal which is not a union of linear subspaces and a computer-assisted proof of a classification of cohomological support varieties of equigenerated monomial ideals over $\mathbb{Q}$ with 6 generators.
\end{abstract}
\maketitle
\section{Introduction}
Recent work \cite{bounds,embdef} has considered the realizability of cohomological support varieties in various levels of generality. Let $Q$ be a regular local ring with residue field $k$ or polynomial ring over $k$, $R=Q/I$ be a quotient by $I=(\boldsymbol{\mathit{f}})=(f_1,\ldots,f_n)$ and $M$ be an finitely generated $R$-module. The \textit{cohomological support variety} $\operatorname{V}_R(M)$, born of the Hom complex $\operatorname{Ext}_E^\bullet(M,k)$ over a certain DG algebra $E$, is a valuable homological invariant which is intimately related to a number of properties of both $M$ and $R$ \cite{avramov,ab00supp,particle,bounds}.

The set of possible values of $\operatorname{V}_R(M)$, the subject of the \textit{realizability problem} for cohomological support varieties, has been classified when $R$ is a complete intersection \cite{bergh07} or Golod \cite{bounds}. Work in \cite{embdef} restricts attention to the case $M=R=Q/(\boldsymbol{\mathit{f}})$ and classifies $\operatorname{V}_R(R)$ for monomial ideals where $\boldsymbol{\mathit{f}}$ has at most 5 generators, provided that either $\widehat{R}=Q/I$ where $Q$ is regular local with residue field $k$ and $I$ is in the square of the maximal ideal of $Q$ or $Q$ is a graded polynomial ring with base field $k$ and $I$ is generated by forms of degree at least two, which is to say that $R$ has a \textit{minimal regular presentation}:
\begin{theorem}[{\cite[Theorem 6.14, Theorem 6.16]{embdef}}]
Let $Q/I$ be a minimal regular presentation of $R$ and let $\boldsymbol{\mathit{f}}$ be a minimal generating set of $I$ consisting of $n\leq5$ monomials. Then $\operatorname{V}_R(R)$ is either a coordinate subspace of $\mathbb{A}^5_k$ or a union of two hyperplanes.
\end{theorem}

\noindent Using a Macaulay2 calculation, they find a cohomological support variety of a monomial ideal which is not a union of linear subspaces which is realizable when $n=6$ \cite[Example 6.17]{embdef}. In doing so, they employ a procedure which can be used to compute the cohomological support variety of an arbitrary monomial ideal. However, this procedure involves the computation of the homology of a square matrix of dimension equal to the sum of the Betti numbers of the ideal (which, for example, will be $2^n$ whenever for each generator there is some variable such that the power of that variable is largest at that generator \cite{mintaylor}), making it relatively manually intractable and potentially computationally expensive. For certain monomial ideals, including those which are equigenerated (that is, whose minimal generating sets consist of monomials of the same degree, that is, they are \textit{equidegree}), we offer a more efficient procedure. This procedure considers a $2^n$-dimensional space, but rather than calculating the homology of an arbitrary automorphism, allows us to partition the space into grades and consider an automorphism which is a differential of a chain complex, which allows for the computation of the homologies of smaller matrices:
\begin{introthm}[Corollary \ref{theoremareal}]\label{thma}
Every cohomological support variety of a ring with a minimal regular presentation given by an equigenerated monomial ideal with $n$ generators is the set of points in $\mathbb{A}_k^n$ such that a chain complex of vector spaces with total dimension $2^n$ with entries defined by polynomials in the $n$ variables begetting our affine space has non-trivial homology.
\end{introthm}

This consequently allows us to theoretically verify their computational discovery and extend it to a second such variety:

\begin{introthm}[{cf. \cite[Example 6.17]{embdef}}]\label{617}
If
\[R=k[x_1,\ldots,x_6]/(x_1x_2,x_2x_3,x_3x_4,x_4x_5,x_5x_6,x_6x_1),\]
\noindent then
\[
\operatorname{V}_R(R)=\mathcal{V}(a_1a_3a_5+a_2a_4a_6)\subseteq\mathbb{A}_k^6.
\]
Furthermore, if $\operatorname{char}(k)\not\in\{2,5\}$ and
\[R=k[x_1,\ldots,x_{10}]/(x_1x_2,x_2x_3,\ldots,x_{10}x_1),\]
\noindent then
\[
\operatorname{V}_R(R)=\mathcal{V}(a_1a_3a_5a_7a_9+a_2a_4a_6a_8a_{10})\subseteq\mathbb{A}_k^{10}.
\]
\end{introthm}

\noindent It also allows us to write a method in Macaulay2 which computes cohomological support varieties of monomial ideals over $\mathbb{Q}$. This allows us to make two additional statements, albeit verified only with computer assistance. First, we are able to give another example of a previously-unknown realizable variety, given by the edge ideal of a 14-cycle over $\mathbb{Q}$:
\begin{introcomp}[Computation \ref{code1}]\label{code11}
The edge ideal of a 14-cycle over $\mathbb{Q}$ has support variety
\[\mathcal{V}(a_1a_3\cdots a_{13}+a_2a_4\cdots a_{14}).\]
\end{introcomp}
\noindent The second is a partial generalization of Theorem \ref{onefour} to the case $n=6$, under the condition that we only consider equigenerated monomials and only over $\mathbb{Q}$:
\begin{introcomp}[Computation \ref{code2}]\label{code21}
The cohomological support varieties of rings over $\mathbb{Q}$ with minimal regular presentations given by ideals with equigenerated, 6-element minimal generating sets are all one of the following up to order:
\begin{itemize}
\item a linear subspace,
\item a union of two hyperplanes,
\item $\mathcal{V}(a_{135}+a_{246})$.
\end{itemize}
\end{introcomp}

In Section \ref{bkgd}, we provide a synopsis of the relevant background information on cohomological support varieties and their calculation in the general and monomial settings, introducing also for the sake of resolving monomial ideals the Taylor resolution. In Section \ref{otimesk}, we will extrapolate on an existing construction of the cohomological support variety in the monomial setting, leveraging the fact that our resulting resolution is over $\overline{k}$ and exploiting the relationship between the Taylor resolution and cellular cochain complexes of certain simplicial complexes, whence Theorem \ref{thma} is derived. In Section \ref{edge} we use this information to manually verify Theorem \ref{617}. In Section \ref{compimp} we describe our computational verification of Computations \ref{code11} and \ref{code21}. We conclude with Section \ref{bgpfuture}, a short note on future work.

\section{Background}\label{bkgd}

We provide a synopsis of the construction of cohomological support varieties, including the relevant topics surrounding DG algebra and module resolutions, Koszul resolutions, and universal resolutions. This section draws predominantly from \cite{particle,embdef}. The interested reader is encouraged to consult \cite{avramov} for more background on DG algebras and resolutions and \cite{pollitz21} for Koszul resolutions, universal resolutions, and the resulting constructions of the cohomological support variety.

\subsection{Preliminary notions}
Let $Q$ be a commutative ring and let $R=Q/I$ where $I=(f_1,\ldots,f_n)$. $E$ is the DG algebra
\[Q[e_i|\partial e_i=f_i],\]
that is, it is the Koszul complex of $(\boldsymbol{\mathit{f}})$ endowed with an exterior algebra structure. In particular, note that these $e_i$ are anti-commutative. 

Let $A$ be a DG $Q$-algebra and $B$ be a DG $A$-module. $A^\natural$ is the underlying graded $Q$-algebra of $A$ and $B^\natural$ is the underlying graded $A^\natural$-module of $B$. Furthermore, the \textit{suspension} $\mathsf{\Sigma}^iB$ of a DG $A$-module $B$ for some integer $i$ is given by
\begin{align*}
(\mathsf{\Sigma}^iB)_n=B_{n-i},\quad\partial^{\mathsf{\Sigma}^iB}=(-1)^i\partial^B\quad\text{and}\quad a\cdot \mathsf{\Sigma}^ib=(-1)^{|a|i} \mathsf{\Sigma}^i(ab).
\end{align*}
We define the suspension of a chain complex which is not a DG module by imposing only the first two of these three conditions.

A DG $A$-module $P$ is \textit{semiprojective} if for every morphism of DG $A$-modules $\alpha:P\to N$ and each surjective quasi-isomorphism of DG A-modules $\gamma:M\to N$ there exists a unique-up-to-homotopy morphism $\beta:P\to M$ of DG $A$-modules such that $\alpha=\gamma\beta$. A DG $A$-module $F$ is \textit{semifree} if $F^\natural$ is a free $A^\natural$-module. A \textit{semiprojective (resp. semifree) resolution} is a quasi-isomorphism $P\to M$ of DG $A$-modules such that $P$ is semiprojective (resp. semifree). Considering $E$ in particular, a \textit{Koszul resolution} of a DG $E$-module $M$ is a map of DG $E$-modules $P\to M$ which, as a map of DG $Q$-modules via restriction of scalars $Q\to E$, is a semiprojective resolution.

Let $M,N$ be DG $A$-modules. $\operatorname{Ext}_A(M,N)$ is the graded $Q$-module with components
\[
\operatorname{Ext}_A^i(M,N)=\operatorname{H}^i(\operatorname{Hom}_A(F,N)),
\]
where $F\xrightarrow{\simeq}M$ is a semifree resolution of $M$ over $A$. Since elements of $\operatorname{Ext}_A^i(M,N)$ are homology classes of $A$-linear chain maps $\alpha:F\to \mathsf{\Sigma}^iN$, we denote them $[\alpha]$. As any two semifree resolutions are unique up to homotopy equivalence and as $\operatorname{Hom}_A(F,-)$ preserves surjections and quasi-isomorphisms, $\operatorname{Ext}_A(M,N)$ is independent of $F$. We concern ourselves with the composition pairing
\[
\operatorname{Ext}_A(M,N)\otimes_Q\operatorname{Ext}_A(L,M)\to \operatorname{Ext}_A(L,N)\quad\text{given by }([\alpha],[\beta])\mapsto  \left[\mathsf{\Sigma}^{|\beta|}\alpha\circ \tilde{\beta}\right]
\] 
where $L$, $M$, and $N$ are DG $A$-modules and $\tilde{\beta}$ satisfies the commutative diagram
\begin{center}
    \begin{tikzcd}
        & \mathsf{\Sigma}^{|\beta|}G\ar[d,"\simeq"]\\
        F\ar[r,swap,"\beta"]\ar[ur,"\tilde{\beta}"]& \mathsf{\Sigma}^{|\beta|}M 
    \end{tikzcd}
\end{center}
with  $F\xrightarrow{\simeq}L$ and $G\xrightarrow{\simeq} M$ semifree resolutions over $A$. This pairing endows $\operatorname{Ext}_A(M,M)$ with a graded $Q$-algebra structure, and each $\operatorname{Ext}_A(M,N)$ with a graded $\operatorname{Ext}_A(N,N)$-$\operatorname{Ext}_A(M,M)$ bimodule structure.

\subsection{Cohomological support varieties}
Henceforth in this work, we will consider the two following cases simultaneously:
\begin{itemize}
\item $\widehat{R}=Q/I$ with $Q$ a regular local ring and $I$ in the square of the maximal ideal of $Q$ (the \textit{local case}),
\item $R = Q/I$ where $Q$ is a positively graded polynomial algebra over a field and $I$ an ideal generated by homogeneous forms of degree at least 2 (the \textit{graded case}).
\end{itemize}
\noindent In such cases we will call $Q/I$ a \textit{minimal regular presentation} of $R$. Note that in either case we have a residue (respectively, base) field $k$. A cohomological support variety is a homological invariant defined in \cite{vpd} which can be found in many forms \cite{avramov,ab00, ab00supp, ag02} and has been used to prove various homological properties \cite{avramov,ab00supp,nonproxy,bounds,particle,ste14}. It concerns itself with the graded $\operatorname{Ext}_E(k,k)$-module $\operatorname{Ext}_E(M,k)$. More specifically, it is concerned with $\operatorname{Ext}_E(M,k)$ as a module over $\mathcal{S}=k[\chi_1,\ldots,\chi_n]$ (where each $\chi_i$ has homological degree $-2$) which has a natural inclusion into $\operatorname{Ext}_E(k,k)$ \cite[Section 2]{ab00}. In these cases, we have
\[\operatorname{Ext}_E(M,k)\cong\operatorname{H}\left(\mathcal{C}_E(F)\right)\]
where
\begin{equation}\label{preC}
\mathcal{C}_E(F)^\natural:= \mathcal{S}\otimes_k \operatorname{Hom}_Q(F,k)\quad\text{with}\quad\partial^{\mathcal{C}_E(F)}=1\otimes \partial^{\operatorname{Hom}_Q(F,k)}+\sum_{i=1}^n \chi_i\otimes e_i \,.
\end{equation}
\begin{definition}
The \textit{cohomological support variety} $\operatorname{V}_R(M)$ of a module $M$ is the support of the projectivization  of the $\mathcal{S}$-module $\operatorname{Ext}_E(M,k)$:
\[\operatorname{V}_R(M):=\textit{Supp}_{\mathcal{S}}^+\left(\operatorname{Ext}_E(M,k)\right).\]
\end{definition}
We will also use $\operatorname{V}_R(M)$ to refer to the support of $\operatorname{Ext}_E(M,k)\otimes_k\overline{k}$, where $\overline{k}$ is an algebraic closure of $k$ which we fix henceforth—these are determinable from each other by the Nullstellensatz. It is also worth noting that by \cite[\nopp 1.2.2]{embdef}, if $\operatorname{H}(M)$ is finitely generated over $R$, there exists some such $F$ which is a bounded complex of free $Q$-modules, implying that $\mathcal{C}_E(F)$ is a finite rank free graded $\mathcal{S}$-module, and furthermore that $\operatorname{Ext}_E(M,k)$ is finitely generated over $\mathcal{S}$.
\begin{remark}
Say that $M$, $N$ are $R$-modules, which may be considered as DG $E$-modules by considering their multiplication by our $e_i$ terms to be zero. If one replaces $\operatorname{Hom}_Q(F,k)$ with $\operatorname{Hom}_Q(F,N)$ in (\ref{preC}), we get an expression of $\operatorname{Ext}_E(M,N)$ which can be found in \cite{ag02}. If $R$ is a complete intersection, then since $E\to R$ is a quasi-isomorphism this is isomorphic to $\operatorname{Ext}_R(M,N)$. This is how Macaulay2 \cite{M2} calculates \texttt{Ext(M,N)}. A more general version of (\ref{preC}) where $M$ and $N$ can be any DG $E$-module can be found in \cite{pollitz21}.
\end{remark}

Recent work has endeavored to investigate how much information about $M$ and $R$ is encoded in the cohomological support variety. This comprises two questions:
\begin{itemize}
\item What values can $\operatorname{V}_R(M)$ take on? (This is known as the \textit{realizability question} for cohomological support varieties)
\item What do those values tell us about $M$ and $R$?
\end{itemize}
\begin{theorem}[{\cite[Theorem B]{pollitz21}}]
If $Q$ is a regular local ring, $R$ is a complete intersection if and only if $\operatorname{V}_R(R)=\varnothing$.
\end{theorem}
\noindent Let $\operatorname{P}_k^R(t)$ denote the Poincar\'{e} series of $M$ over $R$. By, for example, \cite[Proposition 3.3.2]{avramov}, every coefficient of the Poincar\'{e} series $\operatorname{P}_k^R(t)$ is bounded above by the corresponding coefficient of the power series
\[\frac{(1+t)^{\operatorname{edim}R}}{1-\sum_{j=1}^{\operatorname{codepth}R}\operatorname{rank}_kH_j(K^R)t^{j+1}},\]
where edim is the embedding dimension and $K^R$ is the Koszul complex on a minimal generating set of $\mathfrak{m}R$. We say $R$ is \textit{Golod} if this equality holds at each entry for $M=k$.
\begin{theorem}[{\cite[Theorem D]{bounds}}]
Let R be a Golod local ring. For any bounded complex $M$ of finitely generated $R$-modules with $\operatorname{H}(M)=0$, expressed as an $E$-module by assigning to each $e_i$ the trivial action, the cohomological support variety $\operatorname{V}_R(M)$ is either all of $\mathbb{A}_k^n$ or a (conical) hypersurface, and every hypersurface is a cohomological support variety of some such complex. Furthermore, if $R$ is not a hypersurface ring, then $\operatorname{V}_R(R)=\mathbb{A}_k^n$.
\end{theorem}
We will consider \textit{monomial ideals} within these $Q$, referring to monomials on a regular sequence of $Q$ or honest-to-God monomials, respectively. The consideration of these ideals began in \cite{embdef}, and yielded the following
\begin{theorem}[{\cite[Theorem 6.16]{embdef}}]\label{onefour}
Let $R$ have minimal regular presentation $Q/I$ and assume $I$ is minimally generated by five monomials. Then the cohomological support variety of $R$ is either a coordinate subspace of $\mathbb{A}^5_k$ or a union of two hyperplanes. More precisely, up to reordering of the generators, $\operatorname{V}_R(R)$ is either $\mathcal{V}(\chi_1\chi_5)$ or a hyperplane subspace.
\end{theorem}

\subsection{Constructing \texorpdfstring{$\bm{{\operatorname{V}}_R(R)}$}{VR(R)} for \texorpdfstring{$\bm{R}$}{R} a monomial ideal}
We provide a blueprint for constructing $\operatorname{V}_R(R)$ for $R$ a monomial ideal, based primarily on \cite{embdef}, rather than a comprehensive account on how the underlying structures can be used to constructing cohomological support varieties in a more general setting. Those interested in performing these constructions in a more general setting should consult \cite{particle}.

Our calculations of the cohomological support varieties of monomial ideals in this paper stem from the following
\begin{proposition}[{\cite[Proposition 2.8]{embdef}}]\label{prop_supp_Chat}
For an $R$-complex $M$ with $H(M)$ finitely generated, fix a Koszul resolution $F$ of $M$. Then
\[
\operatorname{V}_R(M)=\{a\in \mathbb{A}_{\overline{k}}^n: H(\widehat{\mathcal{C}}_{E_a}(F))\neq 0\} \cup\{0\}=\{a\in \mathbb{A}_{\overline{k}}^n: H_p(\widehat{\mathcal{C}}_{E_a}(F))\neq 0\} \cup\{0\}\,,
\]
where $p\in \{\mathrm{even},\mathrm{odd}\}$ and
\begin{equation}
\label{chat}\widehat{\mathcal{C}}_{E_a}(F):=\cdots \longrightarrow  F_{\rm even}\otimes_Q  \overline{k}\xrightarrow{\ d_a\ }    F_{\rm odd}\otimes_Q \overline{k}\xrightarrow{\ d_a\ }
    F_{\rm even}\otimes_Q \overline{k}\longrightarrow \cdots
\end{equation}
where
\[
d_a:=\partial^{F}\otimes 1+\sum e_i \otimes a_i
\]
and $\overline{k}$ is a fixed algebraic closure of $k$.
\end{proposition}
Letting $B_{\text{even}}$ and $B_{\text{odd}}$ be bases of $F_{\text{even}}\otimes_Q\overline{k}$ and $F_{\text{odd}}\otimes_Q\overline{k}$ respectively, the above describes $d_a$ as an automorphism on $F\otimes_Q\overline{k}$ expressible as follows:
\[
\begin{blockarray}{ccc}
& B_{\text{even}} & B_{\text{odd}}\\
\begin{block}{c[cc]}
  B_{\text{even}} & 0 & * \\
  B_{\text{odd}} & * & 0 \\
\end{block}
\end{blockarray}.
\]
This formulation leads us to the following
\begin{corollary}\label{ceafp}
Under the conditions above,
\[
\operatorname{V}_R(M)=\{a\in \mathbb{A}_{\overline{k}}^n: H(\widehat{\mathcal{C}}_{E_a}'(F))\neq 0\} \cup\{0\}\,,
\]
where $\widehat{\mathcal{C}}_{E_a}'(F)$ is the vector space $F\otimes_Q\overline{k}$ endowed with the automorphism given by $d_a$.
\end{corollary}

\noindent This proposition allows us to explicitly determine the cohomological support variety of an $R$-module $M$, given a sufficiently simple Koszul resolution $F$ of $M$.

$\widehat{\mathcal{C}}_{E_a}'(F)$ can be shown 

We now introduce some prerequisites specific to monomial ideals. To start, henceforth in this chapter and the next, let $Q/I$ be a minimal regular presentation of $R$ with residue or base field $k$, depending on whether we lie in the local case or the graded case, respectively. We will let $d$ be the embedding dimension of $Q$ and fix a regular sequence $x_1,\ldots,x_d$ of $Q$ in the former case, and name our variables as such in the latter, so that the notion of a \textit{monomial} of $Q$ is well-defined. We let $\boldsymbol{\mathit{f}}=f_1, \ldots, f_n$ be a sequence of monomials in $Q$ such that $I = (f_1, \ldots, f_n)$ is a proper ideal of $Q$.

\begin{notation}\label{signnotation}
Let $[n]$ denote the ordered set $\{ 1, \ldots, n \}$. If $J$ is a set of integers, ordered or unordered, we let $|J|$ denote its cardinality and $\operatorname{sort}(J)$ denote the ordered set with the same elements as $J$, ordered from least to greatest. If $J$ is ordered, $\operatorname{sgn}(J)$ denote the sign of the permutation of $J$ yielding $\operatorname{sort}(J)$. If $J$ and $K$ are disjoint sets of integers, we let $JK$ denote their concatenation.
\end{notation}

\noindent The following construction can be traced back to \cite{taylorthesis}:

\begin{construction}\label{taylor}
Each $f_j$ can be written as
\[
f_j = x_1^{a_{j1}} \cdots x_d^{a_{jd}} \quad \text{ for some }a_{ji} \geqslant 0\,.
\]
Given a subset $J$ of $[n]$, set
\[
f_J:= x_1^{a_{J1}}\cdots x_d^{a_{Jd}}\quad\text{where }a_{Ji}=\max\{a_{ji}:j\in J\}\,.
\]
In particular, $f_j=f_{\{j\}}$.
The \textit{Taylor complex on $\boldsymbol{\mathit{f}}$} (with respect to $x_1,\ldots,x_d$) is a  free $Q$-complex $T=T(\boldsymbol{\mathit{f}})$, defined as follows. As a free graded $Q$-module, $T^\natural$ can be assigned a basis in degree $i$ given by
\[
\{b_J': J\subseteq[n]\text{ with }|J|=i\}\,,
\]
and if $J$ is an ordered set we write $b_{J}'=\operatorname{sgn}(J)b_{\operatorname{sort}(J)}'$. The differential on $T$ is the $Q$-linear map determined by 
\[
\partial(b_J') = \sum_{i=1}^s (-1)^{i-1} \, \frac{f_J}{f_{J \backslash \{ j_i \}}} \, b_{J \backslash \{ j_i \}}'\quad \text{with }J=\{j_1 < \cdots < j_s\}\,.
\]
We equip the Taylor complex with a $Q$-bilinear product, defined on basis elements by
\[
b_J'\cdot b_K' = {\operatorname{sgn}(JK)} \,  \frac{f_J f_{K}}{f_{J\cup K}} \, b_{J\cup K}'\,.
\]
Note that 
\[
b_J'\cdot b_K'=0 \quad \text{if } J \cap K \neq \varnothing.
\]
\end{construction}
\begin{proposition}[{\cite[Theorem 12]{taylorthesis}}]
$T(\boldsymbol{\mathit{f}})$ under these conditions (namely, when $\boldsymbol{\mathit{f}}$ is a sequence of monomials) is a DG $Q$-algebra resolution of $R$. 
\end{proposition}

\noindent Consequently, we will refer to $T(\boldsymbol{\mathit{f}})$ as the \textit{Taylor resolution of $R$ via $\boldsymbol{\mathit{f}}$}, or simply the \textit{Taylor resolution} of $R$ when $\boldsymbol{\mathit{f}}$ is clear. A more general construction of the Taylor complex in which $\boldsymbol{\mathit{f}}$ is not necessarily a sequence of monomials and which is not necessarily a resolution can be found, for example, in \cite{yuzvinsky}.
\begin{proposition}
$T(\boldsymbol{\mathit{f}})$ is naturally an $E$-module by letting $e_i\times-$ be given by $b_i\times-$\,, and is as such a Koszul resolution of $R$.
\end{proposition}
\noindent As such, we are free to use $T(\boldsymbol{\mathit{f}})$ as our Koszul resolution $F$ of $R$ in our calculations of $\operatorname{V}_R(R)$.
\begin{remark}[The Taylor resolution versus the Koszul complex]
The Taylor resolution is a modification of the Koszul complex in the monomial ideal case which ensures a resolution. In the Koszul complex, if we assign to each copy of $Q$ the product of the corresponding generators of our ideal, then our differential maps ``preserve'' these assignments in the sense that
\begin{equation}
\label{assignments}\text{assignment}_\text{source}\times\text{ coefficient of map}=\text{assignment}_\text{target}\,,
\end{equation}
allowing for a resolution when our generators form a regular sequence. We observe this same mantra when constructing the Taylor resolution, but in order to construct a resolution when our sequence may not be regular, our assignments of our copies of $Q$ ought to correspond to the least common multiples (LCMs) of the corresponding generators of our ideal, rather than their product.
\end{remark}

Note that our construction of $\operatorname{V}_R(R)$ only requires us to understand $T(\boldsymbol{\mathit{f}})\otimes \overline{k}$. The structure of $T(\boldsymbol{\mathit{f}})\otimes \overline{k}$ is much simpler than that of $T(\boldsymbol{\mathit{f}})$ alone:
\begin{construction}[{$\boldsymbol{\mathit{T(f)\otimes \overline{k}}}$}]\label{taylork}
Let $a_{ji}$ be as in Construction \ref{taylor}. As a free graded $\overline{k}$-module, $(T(\boldsymbol{\mathit{f}})\otimes \overline{k})^\natural$ can be assigned a basis in degree $i$ given by
\[
\{b_J: J\subseteq[n]\text{ with }|J|=i\}\,,
\]
and if $J$ is an ordered set we write $b_{J}=\operatorname{sgn}(J)b_{\operatorname{sort}(J)}$. The differential on $T(\boldsymbol{\mathit{f}})\otimes \overline{k}$ is the $\overline{k}$-linear map determined on basis elements by 
\[
\partial(b_J) = \sum_{f_J=f_{J\backslash\{j_i\}}} (-1)^{i-1} \, b_{J \backslash \{ j_i \}}\quad \text{where }J=\{j_1 < \cdots < j_s\}.
\]
We equip $T(\boldsymbol{\mathit{f}})\otimes \overline{k}$ with a $\overline{k}$-bilinear product, defined on basis elements by
\[
b_J\cdot b_K = \begin{cases}{\operatorname{sgn}(JK)} \, b_{J\cup K}&\text{ if }f_Jf_K=f_{J\cup K}\text{ and }J\cap K=\varnothing,\\
0&\text{otherwise.}
\end{cases}
\]
\end{construction}

\section{Decomposing \texorpdfstring{$T(\boldsymbol{\mathit{f}})\otimes \overline{k}$}{T(f) tensor k-bar}}\label{otimesk}
\subsection{Our decomposition}

The complex $T(\boldsymbol{\mathit{f}})\otimes \overline{k}$ can be decomposed into subcomplexes. In this section, we will use the following
\begin{example}\label{edgeex}
Let $Q=\mathbb{Q}[x_1,\ldots,x_7]$, $\boldsymbol{\mathit{f}}=(x_1x_2,x_2x_3,\ldots,x_6x_7,x_7x_1)$, and $J=\{1,3,5\}$.
\end{example}

\begin{definition}
Let
\[S_J=\left\{K\ \middle|\ f_K=f_J\right\},\]
let $T_J(\boldsymbol{\mathit{f}})$ be the restriction of $T(\boldsymbol{\mathit{f}})\otimes \overline{k}$ to the basis $b_K$ with $K\in S_J$, and let $M_J=\bigcup_{K\in S_J} K$.
\begin{lemma}\label{feb26}
$M_J=M_{J'}$ if and only if $f_J=f_{J'}$.
\end{lemma}
\begin{proof}
Say that $M_J=M_{J'}$. $M_J$ is the union of a positive number of sets $K$ such that $f_K=f_J$, so, since the LCM of a union of sets is the LCM of the LCMs of those sets, $f_{M_J}$ is the LCM of many copies of $f_J$, and is thus itself $f_J$. Thus we also have $f_{J'}=f_{M_J'}=f_{M_J}$. Conversely, if $f_J=f_{J'}$, then $S_J=S_{J'}$, so $M_J=M_{J'}$.
\end{proof}
\begin{lemma}
$M_J=\{j\in[n]\ |\ f_j|f_J\}$.
\end{lemma}
\begin{proof}
If $f_j|f_J$, then $f_{J\cup\{j\}}=f_J$, so $j\in M_J$. If $j\in M_J$, then there exists some $K\ni j$ such that $f_j|f_K=f_J$.
\end{proof}
\end{definition}
\begin{remark}
In Example \ref{edgeex}, we have
\begin{align*}
f_J&=x_1x_2x_3x_4x_5x_6,\\
S_J&=\{\{1,3,5\},\{1,3,4,5\},\{1,2,4,5\},\{1,2,3,5\},\{1,2,3,4,5\}\},\text{ and}\\
M_J&=\{1,2,3,4,5\},\\
T_J(\boldsymbol{\mathit{f}})&=\underset{\vrule width0pt height4em depth0pt 5}{\overline{k}}\xrightarrow[\renewcommand*{\arraystretch}{.6}\scalebox{.7}{$\begin{bmatrix}-1\\1\\-1\end{bmatrix}$}]{}\underset{\vrule width0pt height4em depth0pt 4}{\overline{k}}^3\xrightarrow[\renewcommand*{\arraystretch}{.6}\scalebox{.7}{$\begin{bmatrix}1&0&-1\end{bmatrix}$}]{}\underset{\vrule width0pt height4em depth0pt 3}{\overline{k}}\ ,
\end{align*}
where we assign the three non-trivial vector spaces in $T_J(\boldsymbol{\mathit{f}})$ bases
\[\left(b_{\{1,2,3,4,5\}}\right), \left(b_{\{1,3,4,5\}},b_{\{1,2,4,5\}},b_{\{1,2,3,5\}}\right),\left(b_{\{1,3,5\}}\right),\]
respectively.
\end{remark}
\begin{proposition}
$T_J(\boldsymbol{\mathit{f}})$ is a subcomplex of $T(\boldsymbol{\mathit{f}})\otimes\overline{k}$, and furthermore $T(\boldsymbol{\mathit{f}})\otimes\overline{k}$ is the direct sum of $T_{M_J}(\boldsymbol{\mathit{f}})$ over all distinct $M_J$.
\end{proposition}
\begin{proof}
Recall that $T_{M_J}(\boldsymbol{\mathit{f}})$ is the restriction of $T(\boldsymbol{\mathit{f}})\otimes\overline{k}$ to the basis $b_K$ with $K\in S_J$. Note that for each $K\subseteq[n]$ there is exactly one set $S_J$ containing $K$, since the set of distinct $S_J$ is a partition of the sets $K\subseteq [n]$ by the values $f_K$. Thus it suffices to show that, if we consider the differential of $T(\boldsymbol{\mathit{f}})\otimes\overline{k}$ as a matrix by using the basis $\{b_K\ |\ K\subset [n]\}$, then whenever $M_J\neq M_{J'}$, which by Lemma \ref{feb26} is to say whenever $f_J=f_{J'}$, then the coefficient of this matrix from $b_J$ to $b_J$ is zero. This is clear from the description of the differential in Construction \ref{taylork}.
\end{proof}
\noindent We refer to the complexes $T_J(\boldsymbol{\mathit{f}})$ as \textit{Taylor subcomplexes} on $\boldsymbol{\mathit{f}}$ (with respect to ($x_1,\ldots,x_d$)).

The differential structures of these subcomplexes are those of the cohomologies of corresponding simplicial complexes. Let us be precise. Let $K_\Delta=M_K\,\backslash\,K$ and let $\Delta_J=\left\{K_\Delta\ \middle|\ K\in S_J\right\}$, noting that $K\in S_J$ if and only if $K_\Delta\in \Delta_J$ if and only if $M_J=M_K$.
\begin{lemma}
$\Delta_J$ is a simplicial complex.
\end{lemma}
\begin{proof}
We wish to show that if $K_\Delta\in \Delta_J$, then any subset of $K_\Delta$ is in $\Delta_J$. This is equivalent to saying that if $f_K=f_J$, then for any $K\subset L\subset M_J$ we have $f_L=f_J$. Since $L\supset K$ we have $f_J=f_K|f_L$, and since $L\subset M_J$ we have $f_L|f_{M_J}=f_J$, which together suffice as the LCMs in question are all those of monomials. 
\end{proof}
Consider the augmented chain complex $\widetilde{C}_\bullet(\Delta_J,\overline{k})$. Its $i$-th entry has basis the $i$-simplices $K_\Delta$ of $\Delta_J$, interpreting the trivial simplex as having dimension $-1$. As such, the $i$-th entry of the augmented cochain complex $\widetilde{C}^{\bullet}(\Delta_J,\overline{k})$ is the dual $\operatorname{Hom}(\widetilde{C}_i(\Delta_J,\overline{k}),\overline{k})$, and thus we can assign it a basis consisting of the resulting duals of the $i$-simplices of $\Delta_J$: namely, the dual $(K_{\Delta})^\vee$ of an $i$-simplex $K_\Delta$ sends $K_\Delta$ to 1 and the other $i$-simplices of $\Delta_J$ to zero. We use $c_K'$ to refer to $(K_\Delta)^\vee$. Thus, there is a basis of the $i$-th entry of $\widetilde{C}^\bullet\left(\Delta_J,\overline{k}\right)$ consisting of the terms $c_K'$ over all $K$ such that $K_\Delta$ is an $i$-simplex of $\Delta_J$. We call $\Delta_J$ the \textit{monomial subcomplex} of $J$.
\begin{remark}
In Example \ref{edgeex},
\[\Delta_J=\{\varnothing,\{2\},\{3\},\{4\},\{2,4\}\}=\twoseg{2}{4}\dotseg{3}.\]
The bases of the entries of $\widetilde{C}^\bullet(\Delta_J,\overline{k})$ are as follows:
\begin{align*}
\widetilde{C}^{-1}(\Delta_J,\overline{k}):&\,\left(\varnothing^\vee\right)\\
\widetilde{C}^0(\Delta_J,\overline{k}):&\,\left(\{2\}^\vee,\{3\}^\vee,\{4\}^\vee\right)\\
\widetilde{C}^1(\Delta_J,\overline{k}):&\,\left(\{2,4\}^\vee\right),
\end{align*}
and with this basis, we have
\[\widetilde{C}^\bullet(\Delta_J,\overline{k})=\underset{\vrule width0pt height4em depth0pt -1}{\overline{k}}\xrightarrow[\renewcommand*{\arraystretch}{.6}\scalebox{.7}{$\begin{bmatrix}1\\1\\1\end{bmatrix}$}]{}\underset{\vrule width0pt height4em depth0pt 0}{\overline{k}}^3\xrightarrow[\renewcommand*{\arraystretch}{.6}\scalebox{.7}{$\begin{bmatrix}1&0&-1\end{bmatrix}$}]{}\underset{\vrule width0pt height4em depth0pt 1}{\overline{k}}\ .\]
\end{remark}
If $J$ is an ordered set we let $c_J'=\operatorname{sgn}(J)c_{\operatorname{sort}(J)}'$. Furthermore, let $\operatorname{ksgn}(J)$ (the \textit{cochain complex sign} of $J$) be $-1$ to the power of the number of elements of $J$ at even positions of $\operatorname{sort}(M_J)$ and let
\begin{align*}
\operatorname{ksgn}(j,J)&=\operatorname{ksgn}(J)\operatorname{ksgn}(\{j\}J)\\
c_J&=\operatorname{ksgn}(J)b_J,\\
\end{align*}
\begin{lemma}\label{ksgnrel}
Fix some $J\in[n]$. If for all $i\in\mathbb{Z}$ we let $S_{J,i}$ be an ordered set consisting of the elements of $S_J$ which have size $i$, and choose bases $(c_K)|_{K\in S_{J,i}}$ and $(c_K')|_{K\in S_{J,i}}$ for $(T_J(\boldsymbol{\mathit{f}}))_{i}$ and $\widetilde{C}^{|M_J|-i+1}\left(\Delta_J,\overline{k}\right)$ respectively for all $i\in\mathbb{Z}$, then fixing any $i_0\in \mathbb{Z}$ the matrices
\[\partial_{i_0}(T_J(\boldsymbol{\mathit{f}}))\quad\text{and}\quad\partial^{|M_J|-i_0-1}\left(\widetilde{C}^\bullet\left(\Delta_J,\overline{k}\right)\right)\]
with $\overline{k}$-entries are identical.
\end{lemma}
\noindent Note that we do not draw an isomorphism between $T_J(\boldsymbol{\mathit{f}})$ and $\mathsf{\Sigma}^{|M_J|-1}\widetilde{C}^\bullet\left(\Delta_J,\overline{k}\right)$ directly, as such a shift would introduce additional unwanted signs arising from the use of the suspension operator. However, if we were to shift $\widetilde{C}^\bullet\left(\Delta_J,\overline{k}\right)$ by na\"{i}vely re-indexing its entries instead, we could construct a true isomorphism.
\begin{proof}[Proof of Lemma \ref{ksgnrel}]
Fix some $j\in J$. It suffices to show that, whenever $J$ and $J\,\backslash\,\{j\}$ lie in $S_J$, that is, whenever $f_J=f_{J\,\backslash\,\{j\}}$, the coefficient from $c_J$ to $c_{J\,\backslash\,\{j\}}$ is exactly that from $c_J'$ to $c_{J\,\backslash\,\{j\}}'$. We begin with the first of these coefficients. The coefficient from $b_J$ to $b_{J\,\backslash\,\{j\}}$ in $T_J(\boldsymbol{\mathit{f}})$ is $(-1)^{|\{p\in J|\,p<j\}|}$ by definition, so the coefficient from $c_J$ to $c_{J\,\backslash\,\{j\}}$ is
\begin{align*}
(-1)^{|\{p\in J|\,p<j\}|}\operatorname{ksgn}(J)\operatorname{ksgn}(J\,\backslash\,\{j\})&=(-1)^{|\{p\in J|\,p<j\}|}(-1)^{|\{p\in M_J|p<j\}|}\\
&=(-1)^{|\{p\in J_\Delta|\,p<j\}|}\\
&=(-1)^{|\{p\in (J\,\backslash\,\{j\})_\Delta|\,p<j\}|},
\end{align*}
which is by definition the coefficient from $c_J'$ to $c_{J\,\backslash\,\{j\}}'$ in $\widetilde{C}^\bullet\left(\Delta_J,\overline{k}\right)$.
\end{proof}
\begin{remark}
In Example \ref{edgeex}, we have $\operatorname{ksgn}(\{1,2,3,5\})=-1$: there is a single element of the set \{1,2,3,5\}, namely 2, which has even index in the ordered set $\operatorname{sort}(M_{\{1,2,3,5\}})=(1,2,3,4,5)$. We have
\begin{gather*}
\operatorname{ksgn}(\{1,3,5\})=1,\quad\operatorname{ksgn}(\{1,3,4,5\})=-1,\quad\operatorname{ksgn}(\{1,2,4,5\})=1,\\
\operatorname{ksgn}(\{1,2,3,5\})=-1,\quad\operatorname{ksgn}(\{1,2,3,4,5\})=1.
\end{gather*}
This is reflected in the matrices we already calculated for $T_J(\boldsymbol{\mathit{f}})$ and $\widetilde{C}^\bullet(\Delta_J,\overline{k})$: for some $K\in S_J,j\in K$, the matrix entry taking $b_K$ to $b_{K\,\backslash\,\{j\}}$ is that taking $c_K$ to $c_{K\,\backslash\,\{j\}}$ times $\operatorname{ksgn}(K)\operatorname{ksgn}(K\,\backslash\,\{j\})$.
\end{remark}
\noindent Let us explicitly record the action of our DG multiplication structure on the basis $\{c_J\}$. When non-zero, we have
\[b_jc_J=\begin{cases}\operatorname{ksgn}(j,J)c_{\{j\}J}=\operatorname{ksgn}(j,J)\operatorname{sgn}(\{j\}J)c_{\{j\}\cup J}&\text{if }f_{\{j\}}f_J=f_{\{j\}\cup J},\\
0&\text{otherwise.}
\end{cases}\]
Note that if $f_{\{j\}}f_J=f_{\{j\}\cup J}$, then $f_J\neq f_{\{j\}\cup J}$, which means that this multiplication will never have entries within one of our summands of $T(\boldsymbol{\mathit{f}})$.
\subsection{A double complex structure on \texorpdfstring{$\boldsymbol{\widehat{\mathcal{C}}_{E_a}'(F)}$}{C subscript E subscript a prime of F} for \texorpdfstring{$\boldsymbol{\mathit{f}}$}{f} equidegree}
We begin with some further notation for concision. Let $a_J=\prod_{j\in J}a_j$ and define $x_J$ similarly, where in the latter $J$ may be a multiset. Second, we may omit $(\boldsymbol{\mathit{f}})$ from $T_J(\boldsymbol{\mathit{f}})$ when $\boldsymbol{\mathit{f}}$ is clear. Third, henceforth, we will let underlined strings of hex digits (or integers generally, when their delineation is clear) refer to ordered sets (e.g. $\underline{123}:=(1,2,3)$).

The differential $d_a$ on $\widehat{\mathcal{C}}_{E_a}'(F)$, which is identified with $T(\boldsymbol{\mathit{f}})\otimes_Q\overline{k}$, is the sum of the following:
\begin{itemize}
\item the maps $\partial T_{M_J}(\boldsymbol{\mathit{f}})$ over all distinct $M_J$, our \textit{intra-Taylor-subcomplex maps},
\item and the maps
\begin{align*}
a_jb_j&:T_{M_J}(\boldsymbol{\mathit{f}})\to T_{\{j\}\cup M_J}(\boldsymbol{\mathit{f}}),\\
&\qquad c_K\mapsto a_jb_jc_K=\operatorname{ksgn}(j,J)c_{\{j\}J}
\end{align*}
over all distinct $j,M_J$ such that $\operatorname{gcd}(f_{\underline{j}},f_{M_J})=1$,
our \textit{inter-Taylor-subcomplex maps}.
\end{itemize}
Recall the following
\begin{definition}[{Taylor graph, \cite[Definition 6.3]{embdef}}]
The \textit{Taylor graph} associated with $\boldsymbol{\mathit{f}}$ has vertices the subsets of $[n]$ and directed edges from $J$ to $K$ for $J,K\subset[n]$ exactly when the coefficient of $d_a$ taking $b_J$ to $b_K$ is non-zero.
\end{definition}
\noindent Our existing examples have quite large Taylor graphs, so we use the smaller example from \cite[Example 6.10]{embdef}:
\begin{example}[{\cite[Example 6.10]{embdef}}]\label{taylorg}
Let $n=4$ and
\[\boldsymbol{\mathit{f}} = (x_1x_2,x_2x_3,x_3x_4,x_4x_5).\]
The Taylor graph of $\boldsymbol{\mathit{f}}$ is as follows:
\begin{figure}[H]
    \begin{tikzpicture}[
    scale = 0.9,
  node distance=1cm and 1.5cm,
  every node/.style={circle, draw, minimum size=8mm},
  ->,
  baseline=(current bounding box.center)
]

\node (1234) at (0,0) {\underline{1234}};

\node (123)  at (2,2.5)  {\underline{123}};
\node (124)  at (2,1)  {\underline{124}};
\node (134)  at (2,-.5) {\underline{134}};
\node (234)  at (2,-2) {\underline{234}};

\node (12) at (4,4)  {\underline{12}};
\node (13) at (4,2.75)    {\underline{13}};
\node (14) at (4,1.25)  {\underline{14}};
\node (23) at (4,-.25)    {\underline{23}};
\node (24) at (4,-1.75) {\underline{24}};
\node (34) at (4,-3)   {\underline{34}};

\node (1) at (6,2.5)  {\underline{1}};
\node (2) at (6,1)    {\underline{2}};
\node (3) at (6,-.5)  {\underline{3}};
\node (4) at (6,-2)   {\underline{4}};

\node (empty) at (8,0) {$\varnothing$};

\foreach \x in {124,134} {
  \draw (1234) -- (\x);
}

\foreach \x in {1,2,3,4} {
  \draw (empty)-- (\x);
}

\draw (2)--(24);
\draw (4)--(24);
\draw (4)--(14);
\draw (1)--(14);
\draw (3)--(13);
\draw (1)--(13);
\draw (123)--(13);
\draw (12)--(124);
\draw(34)--(134);
\draw (234)--(24);
\end{tikzpicture}
\caption{Taylor graph for $(x_1x_2,x_2x_3,x_3x_4,x_4x_5)$, cf. \cite[Figure 1]{embdef}}
\label{figone}
\end{figure}
\end{example}

\noindent Note that the non-zero coefficients of $d_a$ taking some $b_J$ to some $b_K$ are each one of two types:
\begin{itemize}
\item Those for which $J$ is the union of $K$ and some element—those given by our intra-Taylor-subcomplex maps. If the vertices of the Taylor graph are placed in columns according to their entries in the Taylor complex, as above, these coefficients are those corresponding to right-facing arrows, which point one unit to the right.
\item Those for which $K$ is the union of $J$ and some element. These are those given by our inter-Taylor-subcomplex maps. If the vertices of the Taylor graph are placed in columns according to their entries in the Taylor complex, as above, these coefficients are those corresponding to left-facing arrows, which point one unit to the left.
\end{itemize}
The subsequent work observes that if we can assign to each basis vector $b_J$ an integer such that every edge in the Taylor graph has the integer assigned to the target one less than the edge assigned to the source, then we can construct a complex whose differential is given by $d_a$ and whose $i$-th entry has basis those vectors assigned $i$. That is, if we can assign to our Taylor graph a \textit{weak grading}, then we will be able to construct such a complex:
\begin{definition}
A \textit{weak grading} (cf. \cite[Section 3]{weakgrading}, which provides a definition of a \textit{weakly-graded} poset analogous to our definition here) of a graph is an assignment of an integer weight to each vertex such that each edge has the weight of its target one less than that of its source. We call such a graph \textit{weakly gradable}.
\end{definition}
\begin{lemma}\label{intermediatewg}
If our Taylor graph has a weak grading $\Sigma$, then $d_a$ begets a chain complex differential on the graded vector space with basis $\{b_J\ |\ J\subset[n]\}$ and for which $b_J$ is assigned the grade $\Sigma_J$ given to it by $\Sigma$.
\end{lemma}
Letting $B_i$ be a basis for the vector space assigned grade $i\in\mathbb{Z}$ under this lemma, our ``chain complex differential'' reveals nothing more than that $d_a$ can be described as an automorphism on $\widehat{\mathcal{C}}_{E_a}'(T(\boldsymbol{\mathit{f}}))$ expressible as follows (thanks to our choice of partition $B_i$):
\[
\begin{blockarray}{ccccccc}
& \cdots & B_2& B_1&B_0&B_{-1}&\cdots\\
\begin{block}{c[cccccc]}
 \vdots & \ddots& \vdots&\vdots&\vdots&\vdots&\ddots \\
 B_2 & \cdots& 0&0&0&0&\cdots \\
 B_1 & \cdots&*&0&0&0&\cdots \\
 B_0 & \cdots&0& *&0&0&\cdots \\
 B_{-1} & \cdots& 0&0&*&0&\cdots \\
 \vdots & \ddots& \vdots&\vdots&\vdots&\vdots&\ddots \\
\end{block}
\end{blockarray}.
\]
\begin{example}
Continuing Example \ref{taylorg}, we can weakly grade our Taylor graph as follows, assigning to each subset of $[n]$ the integer heading the column in which it lies:
\begin{figure}[H]
    \begin{tikzpicture}[
    scale = 0.9,
  node distance=1cm and 1.5cm,
  every node/.style={circle, draw, minimum size=8mm},
  ->,
  baseline=(current bounding box.center)
]

\node[style={draw=none}] at (0,5.25) {\scriptsize{6}};
\node[style={draw=none}] at (2,5.25) {\scriptsize{5}};
\node[style={draw=none}] at (4,5.25) {\scriptsize{4}};
\node[style={draw=none}] at (6,5.25) {\scriptsize{3}};
\node[style={draw=none}] at (8,5.25) {\scriptsize{2}};
\node[style={draw=none}] at (10,5.25) {\scriptsize{1}};
\node (1234) at (8,0) {\underline{1234}};

\node (123)  at (2,4)  {\underline{123}};
\node (124)  at (10,1)  {\underline{124}};
\node (134)  at (10,-.5) {\underline{134}};
\node (234)  at (2,-3.5) {\underline{234}};

\node (12) at (8,1.5)  {\underline{12}};
\node (13) at (4,2.75)    {\underline{13}};
\node (14) at (4,1.25)  {\underline{14}};
\node (23) at (4,-.25)    {\underline{23}};
\node (24) at (4,-1.75) {\underline{24}};
\node (34) at (8,-1.5)   {\underline{34}};

\node (1) at (2,2.5)  {\underline{1}};
\node (2) at (2,1)    {\underline{2}};
\node (3) at (2,-.5)  {\underline{3}};
\node (4) at (2,-2)   {\underline{4}};

\node (empty) at (0,0) {$\varnothing$};

\foreach \x in {124,134} {
  \draw (1234) -- (\x);
}

\foreach \x in {1,2,3,4} {
  \draw (empty)-- (\x);
}

\draw (2)--(24);
\draw (4)--(24);
\draw (4)--(14);
\draw (1)--(14);
\draw (3)--(13);
\draw (1)--(13);
\draw (123)--(13);
\draw (12)--(124);
\draw(34)--(134);
\draw (234)--(24);
\end{tikzpicture}

\vspace{0.5em}
\caption{Weak grading of the Taylor graph for $(x_1x_2,x_2x_3,x_3x_4,x_4x_5)$, cf. \cite[Figure 1]{embdef}}
\end{figure}
\end{example}
\noindent Finding a weak grading of a Taylor graph is not necessarily simple, so we will construct a new, simpler graph through which a weak grading of our Taylor graph can be built:

\begin{definition}
For some set of monomial generators of an ideal, partition the vertices of its Taylor graph by LCM: that is, for each $J\subseteq[n]$, place the vertex corresponding to $J$ in the subset containing $M_J$. The \textit{Taylor subcomplex graph} is the quotient of this Taylor graph by this partition. We use $T_{M_J}(\boldsymbol{\mathit{f}})$ to label the vertex corresponding to the subset containing $M_J$ (so that it shares a name with the Taylor subcomplex to which it corresponds), though we may occasionally refer to this vertex as simply $T_J$.
\end{definition}

\begin{lemma}\label{finalwg}
If our Taylor subcomplex graph has a weak grading $\Sigma$, then $d_a$ begets a chain complex differential on the graded vector space with basis $\{b_J\ |\ J\subset[n]\}$ and for which $b_J$ is assigned the grade $|J|+2\Sigma_{T_J}$.
\end{lemma}
\begin{proof}

By Lemma \ref{intermediatewg} it suffices to show that $J\mapsto|J|+2\Sigma_{T_J}$ is a weak grading of our Taylor graph. Consider an edge from $J$ to $J'$ in our Taylor graph. If it is given by an intra-Taylor-subcomplex map, then $|K|=|K'|+1$ and $\Sigma_{T_J}=\Sigma_{T_{J'}}$. If it is given by an inter-Taylor-subcomplex map, then $|K|=|K'|-1$ and $\Sigma_{T_J}=\Sigma_{T_{J'}}+1$. In either case we have
\[|J'|+2\Sigma_{T_{J'}}=|J|+2\Sigma_{J}-1,\]
which is to say that $J\mapsto|J|+2\Sigma_{T_J}$ is a weak grading of our Taylor graph, completing our proof.
\end{proof}
\noindent We call the chain complex resulting from this lemma $T^\Sigma(\boldsymbol{\mathit{f}},a)$.

\begin{example}\label{taylorsd}
Continuing Example \ref{taylorg}, our partition of vertices preceding our Taylor subcomplex graph would be as follows:
\begin{figure}[H]
    \begin{tikzpicture}[
    scale = 1,
  node distance=1cm and 1.5cm,
  every node/.style={circle, draw, minimum size=8mm},
  ->,
  baseline=(current bounding box.center),outer sep=0pt
]

\node (1234) at (0,0) {\underline{1234}};

\path (1234.east);
\pgfgetlastxy{\xcoord}{\ycoord}
\path (1234.center);
\pgfgetlastxy{\xc}{\yc}
\pgfmathsetmacro{\temp}{sqrt((\xcoord-\xc)^2 + (\ycoord-\yc)^2)}
\expandafter\edef\csname noderadius1234\endcsname{\temp}

\node (123)  at (2,2.5)  {\underline{123}};

\path (123.east);
\pgfgetlastxy{\xcoord}{\ycoord}
\path (123.center);
\pgfgetlastxy{\xc}{\yc}
\pgfmathsetmacro{\temp}{sqrt((\xcoord-\xc)^2 + (\ycoord-\yc)^2)}
\expandafter\edef\csname noderadius123\endcsname{\temp}

\node (124)  at (2,1)  {\underline{124}};

\path (124.east);
\pgfgetlastxy{\xcoord}{\ycoord}
\path (124.center);
\pgfgetlastxy{\xc}{\yc}
\pgfmathsetmacro{\temp}{sqrt((\xcoord-\xc)^2 + (\ycoord-\yc)^2)}
\expandafter\edef\csname noderadius124\endcsname{\temp}

\node (134)  at (2,-.5) {\underline{134}};

\path (134.east);
\pgfgetlastxy{\xcoord}{\ycoord}
\path (134.center);
\pgfgetlastxy{\xc}{\yc}
\pgfmathsetmacro{\temp}{sqrt((\xcoord-\xc)^2 + (\ycoord-\yc)^2)}
\expandafter\edef\csname noderadius134\endcsname{\temp}

\node (234)  at (2,-2) {\underline{234}};

\path (234.east);
\pgfgetlastxy{\xcoord}{\ycoord}
\path (234.center);
\pgfgetlastxy{\xc}{\yc}
\pgfmathsetmacro{\temp}{sqrt((\xcoord-\xc)^2 + (\ycoord-\yc)^2)}
\expandafter\edef\csname noderadius234\endcsname{\temp}

\node (12) at (4,4)  {\underline{12}};
\node (13) at (4,2.75)    {\underline{13}};

\path (13.east);
\pgfgetlastxy{\xcoord}{\ycoord}
\path (13.center);
\pgfgetlastxy{\xc}{\yc}
\pgfmathsetmacro{\temp}{sqrt((\xcoord-\xc)^2 + (\ycoord-\yc)^2)}
\expandafter\edef\csname noderadius13\endcsname{\temp}

\node (14) at (4,1.25)  {\underline{14}};
\node (23) at (4,-.25)    {\underline{23}};
\node (24) at (4,-1.75) {\underline{24}};

\path (24.east);
\pgfgetlastxy{\xcoord}{\ycoord}
\path (24.center);
\pgfgetlastxy{\xc}{\yc}
\pgfmathsetmacro{\temp}{sqrt((\xcoord-\xc)^2 + (\ycoord-\yc)^2)}
\expandafter\edef\csname noderadius24\endcsname{\temp}

\node (34) at (4,-3)   {\underline{34}};

\node (1) at (6,2.5)  {\underline{1}};
\node (2) at (6,1)    {\underline{2}};
\node (3) at (6,-.5)  {\underline{3}};
\node (4) at (6,-2)   {\underline{4}};

\node (empty) at (8,0) {$\varnothing$};

\foreach \x in {124,134} {
  \draw (1234) -- (\x);
}

\foreach \x in {1,2,3,4} {
  \draw (empty)-- (\x);
}

\draw (2)--(24);
\draw (4)--(24);
\draw (4)--(14);
\draw (1)--(14);
\draw (3)--(13);
\draw (1)--(13);
\draw (123)--(13);
\draw (12)--(124);
\draw(34)--(134);
\draw (234)--(24);

  \pgfmathsetmacro{\rbig}{\csname noderadius1234\endcsname+5pt}
  \pgfmathsetmacro{\rsmall}{\csname noderadius124\endcsname+5pt}
  \pgfmathsetmacro{\rmed}{\csname noderadius134\endcsname+5pt}
  \pgfmathsetmacro{\rtwofour}{\csname noderadius24\endcsname+5pt}
  \pgfmathsetmacro{\rthree}{\csname noderadius234\endcsname+5pt}
  \pgfmathsetmacro{\ronethree}{\csname noderadius13\endcsname+5pt}
  \pgfmathsetmacro{\rtwo}{\csname noderadius123\endcsname+5pt}
  \path let
  \p1 = (124)
  in node[draw=none,circle,minimum size=2*\rsmall pt] (small) at (\x1,\y1) {};
  \path let
  \p1 = (1234)
  in node[draw=none,circle,minimum size=2*\rbig pt] (big) at (\x1,\y1) {};
  \path let
  \p1 = (134)
  in node[draw=none,circle,minimum size=2*\rmed pt] (med) at (\x1,\y1) {};
  \path let
  \p1 = (234)
  in node[draw=none,circle,minimum size=2*\rthree pt] (mythree) at (\x1,\y1) {};
  \path let
  \p1 = (24)
  in node[draw=none,circle,minimum size=2*\rtwofour pt] (twofour) at (\x1,\y1) {};
  \path let
  \p1 = (123)
  in node[draw=none,circle,minimum size=2*\rtwo pt] (mytwo) at (\x1,\y1) {};
  \path let
  \p1 = (13)
  in node[draw=none,circle,minimum size=2*\ronethree pt] (onethree) at (\x1,\y1) {};
  \coordinate (c) at (barycentric cs:big=-\csname noderadius124\endcsname,small=\csname noderadius1234\endcsname);
  \coordinate (d) at (barycentric cs:med=-\csname noderadius1234\endcsname,big=\csname noderadius134\endcsname);
  \coordinate (e) at (barycentric cs:onethree=-\csname noderadius123\endcsname,mytwo=\csname noderadius13\endcsname);
  \coordinate (f) at (barycentric cs:twofour=-\csname noderadius234\endcsname,mythree=\csname noderadius24\endcsname);
  \node[draw=none] (one) at (tangent cs:node=big,point={(c)},solution=1) {};
  \node[draw=none] (two) at (tangent cs:node=small,point={(c)},solution=1) {};
  \node[draw=none] (three) at (tangent cs:node=big,point={(d)},solution=2) {};
  \node[draw=none] (four) at (tangent cs:node=med,point={(d)},solution=2) {};
  \node[draw=none] (five) at ($(134.center)!-\rmed pt!90:(124.center)$) {};
  \node[draw=none] (six) at ($(124)!\rmed pt!90:(134)$) {};
  \node[draw=none] (seven) at (tangent cs:node=onethree,point={(e)},solution=1) {};
  \node[draw=none] (eight) at (tangent cs:node=onethree,point={(e)},solution=2) {};
  \node[draw=none] (nine) at (tangent cs:node=mytwo,point={(e)},solution=1) {};
  \node[draw=none] (ten) at(tangent cs:node=mytwo,point={(e)},solution=2)  {};
  \node[draw=none] (eleven) at (tangent cs:node=twofour,point={(f)},solution=1) {};
  \node[draw=none] (twelve) at (tangent cs:node=twofour,point={(f)},solution=2) {};
  \node[draw=none] (thirteen) at (tangent cs:node=mythree,point={(f)},solution=1) {};
  \node[draw=none] (fourteen) at (tangent cs:node=mythree,point={(f)},solution=2) {};
  \draw [-] (one.center) -- (two.center);
  \draw [-] (three.center) -- (four.center);
  \draw [-] (five.center) -- (six.center);
  \pic [-,draw, angle radius=\rbig] {angle=one--1234--three};
  \pic [-,draw, angle radius=\rbig] {angle=one--1234--three};
  \pic [-,draw, angle radius=\rsmall] {angle=six--124--two};
  \pic [-,draw, angle radius=\rsmall] {angle=four--134--five};

  \draw [-] (seven.center) -- (nine.center);
  \draw [-] (eight.center) -- (ten.center);
  \pic [-,draw, angle radius=\ronethree] {angle=eight--13--seven};
  \pic [-,draw, angle radius=\rtwo] {angle=nine--123--ten};

  \draw [-] (eleven.center) -- (thirteen.center);
  \draw [-] (twelve.center) -- (fourteen.center);
  \pic [-,draw, angle radius=\ronethree] {angle=twelve--24--eleven};
  \pic [-,draw, angle radius=\rtwo] {angle=thirteen--234--fourteen};
\end{tikzpicture}
\caption{Taylor graph for $(x_1x_2,x_2x_3,x_3x_4,x_4x_5)$ with a partition by Taylor subcomplex, cf. \cite[Figure 1]{embdef}}
\end{figure}
The resulting Taylor subcomplex graph, to which we have assigned a weak grading, would be as follows:
\begin{figure}[H]
\[\begin{tikzcd}
	{\scriptstyle 3} & {\scriptstyle 2} & {\scriptstyle 1} & {\scriptstyle 0} & {\scriptstyle -1} \\
	{T_\varnothing} & {T_{\underline{1}}} & {T_{\underline{123}}} & {T_{\underline{12}}} & {T_{\underline{1234}}} \\
	& {T_{\underline{3}}} & {T_{\underline{14}}} & {T_{\underline{34}}} \\
	& {T_{\underline{2}}} & {T_{\underline{234}}} \\
	& {T_{\underline{4}}} & {T_{\underline{23}}}
	\arrow[from=2-1, to=2-2]
	\arrow[from=2-1, to=3-2]
	\arrow[from=2-1, to=4-2]
	\arrow[from=2-1, to=5-2]
	\arrow[from=2-2, to=2-3]
	\arrow[from=2-2, to=3-3]
	\arrow[from=2-4, to=2-5]
	\arrow[from=3-2, to=2-3]
	\arrow[from=3-4, to=2-5]
	\arrow[from=4-2, to=4-3]
	\arrow[from=5-2, to=3-3]
\end{tikzcd}\]
\caption{Weak grading of the Taylor subcomplex graph for $(x_1x_2,x_2x_3,x_3x_4,x_4x_5)$, cf. \cite[Figure 1]{embdef}}
\label{itseff}
\end{figure}








\noindent The weak grading of the Taylor graph implied by the above weak grading of our Taylor subcomplex graph via Lemma \ref{finalwg} is that shown in Figure \ref{figone}.
\end{example}

\begin{theorem}\label{threefive}
If our Taylor subcomplex graph has a weak grading $\Sigma$,
\[\operatorname{V}_R(R)=\{a\in \mathbb{A}_{\overline{k}}^n: H(T^\Sigma(\boldsymbol{\mathit{f}},a))\neq 0\} \cup\{0\}\,.\]
Under this condition, using the basis $\{b_J\ |\ J\subseteq [n]\}$, the linear maps in $T^\Sigma(\boldsymbol{\mathit{f}},a)$ can be given collectively by assigning matrices with entries polynomials over $a_1,\ldots,a_n$ to each integer, so that selecting some set of values for each $a_i$ yields the differential maps of the corresponding chain complex $T^\Sigma(\boldsymbol{\mathit{f}},a)$.

\end{theorem}
\begin{proof}
This characterization of $\operatorname{V}_R(F)$ is clear by Corollary \ref{ceafp} as $\widehat{\mathcal{C}}_{E_a}'(T(\boldsymbol{\mathit{f}}))$ and $T^\Sigma(\boldsymbol{\mathit{f}},a)$ are identical as vector spaces with automorphisms, and our matrix assignment is clear from the definition of $d_a$.
\end{proof}





\begin{example}
\noindent Let $d=6$, and let $\boldsymbol{\mathit{f}}=(x_{\underline{12}},x_{\underline{34}},x_{\underline{56}},x_{\underline{135}},x_{\underline{246}})$. We wish to draw the Taylor subcomplex graph of this sequence of monomials. Omitting the Taylor subcomplexes
\[T_{\underline{14}},T_{\underline{15}},T_{\underline{24}},T_{\underline{25}},T_{\underline{34}},T_{\underline{35}},T_{\underline{124}},T_{\underline{125}},T_{\underline{134}},T_{\underline{135}},T_{\underline{234}},T_{\underline{235}}\]
from our graph for cleanliness (as they would all be isolated vertices), the remainder of our Taylor subcomplex graph, along with an attempt to weakly grade it, is as follows:
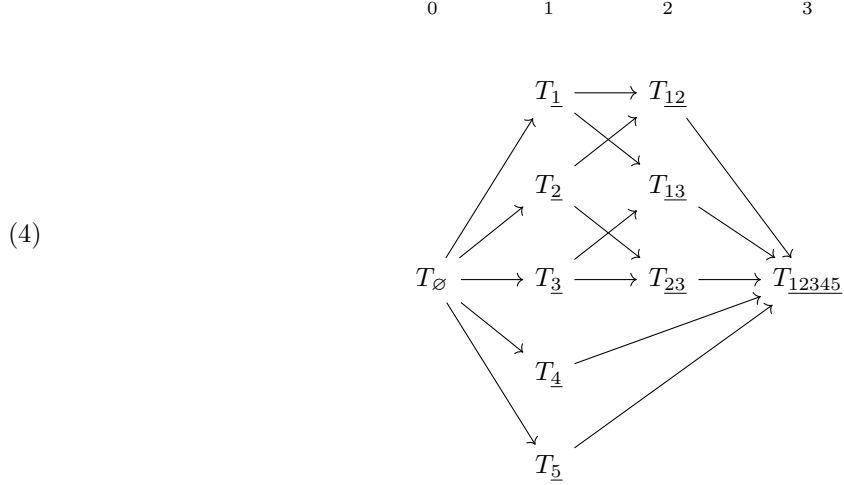
\begin{equation*}
\begin{tikzcd}
	{\scriptstyle 3} & {\scriptstyle 2} & {\scriptstyle 1} & {\scriptstyle 0} \\
	& {T_{\underline{1}}} & {T_{\underline{12}}} \\
	& {T_{\underline{2}}} & {T_{\underline{13}}} \\
	{T_{\varnothing}} & {T_{\underline{3}}} & {T_{\underline{23}}} & {T_{\underline{12345}}} \\
	& {T_{\underline{4}}} \\
	& {T_{\underline{5}}}
	\arrow[from=2-2, to=2-3]
	\arrow[from=2-2, to=3-3]
	\arrow[from=2-3, to=4-4]
	\arrow[from=3-2, to=2-3]
	\arrow[from=3-2, to=4-3]
	\arrow[from=3-3, to=4-4]
	\arrow[from=4-1, to=2-2]
	\arrow[from=4-1, to=3-2]
	\arrow[from=4-1, to=4-2]
	\arrow[from=4-1, to=5-2]
	\arrow[from=4-1, to=6-2]
	\arrow[from=4-2, to=3-3]
	\arrow[from=4-2, to=4-3]
	\arrow[from=4-3, to=4-4]
	\arrow[from=5-2, to=4-4]
	\arrow[from=6-2, to=4-4]
\end{tikzcd}
\end{equation*}
This graph is not weakly gradable: this can be seen by considering the Taylor subcomplexes $T_\varnothing$, $T_{\underline{1}}$, $T_{\underline{12}}$, $T_{\underline{4}}$, and $T_{\underline{12345}}$.
\end{example}

In order to take advantage of our understanding of weak gradings, we prove their existence when $\boldsymbol{\mathit{f}}$ is equidegree:

\begin{lemma}\label{220}
Whenever $\boldsymbol{\mathit{f}}$ is equidegree, its Taylor subcomplex graph admits a weak grading.
\end{lemma}
\begin{proof}
Let $\Sigma_{T_J}=-\left\lfloor\frac{\operatorname{deg}f_J}{\operatorname{deg}f_1}\right\rfloor$. If $f_{\{j\}\cup J}=f_{\{j\}}f_J$, which is to say that there is an edge in our Taylor subcomplex graph from $T_J$ to $T_{\{j\}\cup J}$, then 
\begin{align*}
&\Sigma_{T_{\{j\}\cup J}}=-\left\lfloor\frac{\deg f_{\{j\}\cup J}}{\deg f_1}\right\rfloor=-\left\lfloor\frac{\deg f_{\{j\}\cup J}}{\deg f_1}\right\rfloor=-\left\lfloor\frac{\deg \left(f_{\{j\}}f_J\right)}{\deg f_1}\right\rfloor\\
&\qquad =-\left\lfloor\frac{\deg f_{\{j\}}+\deg f_J}{\deg f_1}\right\rfloor=-\left\lfloor\frac{\deg \left(f_{\{j\}}f_J\right)}{\deg f_1}\right\rfloor=-\left\lfloor\frac{\deg f_{\{j\}}}{\deg f_1}+\frac{\deg f_J}{\deg f_1}\right\rfloor,
\end{align*}
so by our assumption that $\boldsymbol{\mathit{f}}$ is equidegree, $\deg f_{\{j\}}=\deg f_1$ so $\Sigma_{T_{\{j\}\cup J}}=\Sigma_{T_J}-1$, completing our proof. 
\end{proof}

\begin{corollary}\label{theoremareal}
Every cohomological support variety of a ring with a minimal regular presentation given by an equigenerated monomial ideal with $n$ generators is the set of points in $\mathbb{A}_k^n$ such that a chain complex of vector spaces with total dimension $2^n$ with entries defined by polynomials in the $n$ variables begetting our affine space has non-trivial homology.
\end{corollary}
\begin{proof}
The chain complex $T^{\Sigma}(\boldsymbol{\mathit{f}},a)$ given in Theorem \ref{threefive} fits the description of the chain complex in the statement of the claim, and when $\boldsymbol{\mathit{f}}$ is equigenerated our Taylor subcomplex graph admits a weak grading by Lemma \ref{220}.
\end{proof}
The Macaulay2 method which currently computes cohomological support varieties (\cite[\texttt{extKoszul}]{ThickSubcategoriesSource}) does so by calculating the homology of $d_a$ on $\widehat{\mathcal{C}}_{E_a}'(F)$ for $F$ some minimal resolution. Since the Taylor resolution is not necessarily minimal, and we make no explicit claims regarding the distribution of our chain complex, we cannot make claims that our calculations of cohomological support varieties will always consider strictly smaller matrices. However, whenever for each generator of a monomial ideal there is some variable such that the power of that variable is largest at that generator, for example in the ideal
\begin{equation}\label{eqmar7}
(x_1^2x_2,x_2^2x_3,\ldots,x_d^2x_1)
\end{equation}
over $\mathbb{Q}[x_1,\ldots,x_d]$ for $d\geq2$ a positive integer, the Taylor resolution is minimal \cite[Theorem 4.4]{mintaylor}. This would require the existing method to calculate the homology of a $2^d$-by-$2^d$ square matrix, though with some further implementation it could instead calculate kernels and images of $2^{d-1}$-by-$2^{d-1}$ square matrices by leveraging the decomposition of $F\otimes_Q\overline{k}$ into $F_{\text{even}}\otimes_Q\overline{k}$ and $F_{\text{odd}}\otimes_Q\overline{k}$ as in Proposition \ref{prop_supp_Chat}. For our calculations, the consideration of a $2^{d-1}$-by-$2^{d-1}$ matrix would be the worst case, necessitating that our Taylor graph was concentrated in two consecutive grades. For some perspective on this worst case, the ideal \eqref{eqmar7}, using the construction afforded by Theorem \ref{threefive} and Lemma \ref{220}, will yield non-zero entries in all grades from 0 to $d-2\left\lfloor2d/3\right\rfloor$.

We can impose even further structure on our chain complexes by construction. We follow the convention in which double complexes' arrows go down and to the right, with the first entry decreasing by one when following a right-facing arrow, and with the second entry decreasing by one when following a down-facing arrow:

\begin{theorem}\label{doublecomplex}
If our Taylor subcomplex graph has a weak grading $\Sigma$, then $d_a$ begets a double complex differential on the bigraded vector space with basis $\{b_J\ |\ J\subseteq[n]\}$ which assigns $b_J$ bigrade $(|J|+\Sigma_{T_J},\Sigma_{T_J})$, such that the vertical differential is the sum of our inter-Taylor-subcomplex maps and the horizontal differential is the sum of our intra-Taylor-subcomplex maps. Furthermore, the row with grade $i$ is the direct sum of chain complexes identical to $T_J(\boldsymbol{\mathit{f}})$ for $\Sigma_{T_J}=i$ except that the grade of each basis element is decreased by $\Sigma_{T_J}$. Further still, the na\"{i}ve totalization of this double complex (given by taking the sum of the vertical and horizontal differentials, without imposing any further signs) is $T^{\Sigma}(\boldsymbol{\mathit{f}},a)$.
\end{theorem}
\begin{proof}
Note that any coefficient of $d_a$ born of an intra-Taylor-subcomplex map is of bidegree $(1,0)$ under this bigrading, and any coefficient of $d_a$ born of an inter-Taylor-subcomplex map is of bidegree $(0,-1)$ under this bigrading. To show that this is a double complex, one must show that
\[d_{v,a}^2,d_{h,a}^2,d_{v,a}d_{h,a}+d_{h,a},d_{v,a}=0.\]
Note that for any basis vector $b_J$, the images of these three maps are linear sums of disjoint sets of basis vectors, so it suffices to note that their sum, $(d_{v,a}+d_{h,a})^2$, is zero. However, $d_{v,a}+d_{h,a}$ is simply $d_a$, which we already know has square zero.

We quickly address the claim regarding the decomposition of each row into familiar chain complexes. Since we can see that the basis vectors corresponding to the chain complexes $T_J(\boldsymbol{\mathit{f}})$ in question have the correct bigrades, it suffices to show that there are no horizontal differentials between entries not corresponding to subsets sharing an LCM. However, we have already stated that our horizontal differential is a sum only of intra-Taylor-subcomplex maps, which implies this quite clearly.
\end{proof}
\begin{example}
Continuing with Example \ref{taylorg}, the weak grading given in Figure \ref{itseff} yields the double complex in Figure \ref{figm6}, in which each non-zero entry is given by the collection of vertices whose corresponding basis vectors generate it.
\begin{figure}[h]
\begin{tikzpicture}[
    scale = 0.9,
  node distance=1cm and 1.5cm,
  every node/.style={circle, draw, minimum size=8mm},
  ->,
  baseline=(current bounding box.center)
]

\node (1234) at (12,-5) {\underline{1234}};

\node (123)  at (2.6,-4.6)  {\underline{123}};
\node (124)  at (11.65,-7.6)  {\underline{124}};
\node (134)  at (12.35,-8.4) {\underline{134}};
\node (234)  at (3.4,-5.4) {\underline{234}};

\node (12) at (8.65,-7.7)  {\underline{12}};
\node (13) at (5.55,-4.55)    {\underline{13}};
\node (14) at (6.45,-4.55)  {\underline{14}};
\node (23) at (5.55,-5.45)    {\underline{23}};
\node (24) at (6.45,-5.45) {\underline{24}};
\node (34) at (9.35,-8.3)   {\underline{34}};

\node (1) at (5.55,-1.55)  {\underline{1}};
\node (2) at (6.45,-1.55)    {\underline{2}};
\node (3) at (5.55,-2.45)  {\underline{3}};
\node (4) at (6.45,-2.45)   {\underline{4}};

\node (empty) at (6,1) {$\varnothing$};
\node[draw=none] at (3,1) {0};
\node[draw=none] at (9,1) {0};
\node[draw=none] at (12,1) {0};
\node[draw=none] at (3,-2) {0};
\node[draw=none] at (9,-2) {0};
\node[draw=none] at (12,-2) {0};
\node[draw=none] at (9,-5) {0};
\node[draw=none] at (3,-8) {0};
\node[draw=none] at (6,-8) {0};



\foreach \a in {1,2,3,4} {
\foreach \b in {1,2,3} {
	\draw (3*\b+1,3*\a-11)-- (3*\b+2,3*\a-11);
	\draw (3*\a,3*\b-9)-- (3*\a,3*\b-10);
}
}
\node[draw=none] at (1.5,1) {\scriptsize3};
\node[draw=none] at (1.5,-2) {\scriptsize2};
\node[draw=none] at (1.5,-5) {\scriptsize1};
\node[draw=none] at (1.5,-8) {\scriptsize0};
\node[draw=none] at (3,2.5) {\scriptsize4};
\node[draw=none] at (6,2.5) {\scriptsize3};
\node[draw=none] at (9,2.5) {\scriptsize2};
\node[draw=none] at (12,2.5) {\scriptsize1};
\end{tikzpicture}
\caption{Double complex corresponding to the weak grading of the Taylor subcomplex graph of $(x_1x_2,x_2x_3,x_3x_4,x_4x_5)$ given by Figure \ref{itseff}, in which each non-zero entry is given by the collection of vertices whose corresponding basis vectors generate it, cf. \cite[Figure 1]{embdef}}
\label{figm6}
\end{figure}
\end{example}

We will use $T_{M_J}^\Sigma(\boldsymbol{\mathit{f}},a)$ to refer to the row-direct-summand corresponding to $T_{M_J}(\boldsymbol{f})$ mentioned above, though we may drop the $\boldsymbol{\mathit{f}}$ and $a$ from this notation when they are clear or replace $M_J$ with simply $J$. We let $c_J^\Sigma$ refer to the basis element of this double complex afforded by this corollary corresponding to $c_J$.

\begin{corollary}\label{components}
If the Taylor subcomplex graph of $\boldsymbol{\mathit{f}}$ is weakly gradable, $T^\Sigma(\boldsymbol{\mathit{f}},a)$ can be decomposed by partitioning the basis $\{b_J\ |\ J\subseteq[n]\}$ by the connected component of the Taylor subcomplex graph containing $T_J$. This decomposition also applies to the double complex given in Theorem \ref{doublecomplex}.
\end{corollary}

Finally, we provide information regarding whether a weak grading exists for some given sequence of monomials:
\begin{theorem}
The Taylor subcomplex graph of a sequence $\boldsymbol{\mathit{f}}$ admits a weak grading if and only if there are no two subsets $S,T\subset[n]$ such that any pair of elements of $S$ or pair of elements of $T$ are coprime, $f_S=f_T$, and $|S|\neq|T|$.
\end{theorem}
\begin{proof}
If two such subsets exist then a Taylor subcomplex graph is not weakly gradable: in any Taylor subcomplex graph we can get from $T_{\varnothing}$ to $T_{M_S}$ via either $|S|$ or $|T|$ edges.

If no two such subsets exist, we will iteratively assign weights to our $T_J$ for increasing sizes of $J=M_J$ by the procedure which follows. If $T_J$ is not connected to any subcomplex which has already been assigned a weight, assign it a weight of zero. Note that in this case, a new connected component of the induced subgraph $G$ of Taylor subcomplex graph consisting of those subcomplexes already considered is created. If $T_J$ is connected to some subcomplexes which have already been assigned a weight, note that the edges between $T_J$ and these all have $T_J$ as their target. If these source subcomplexes share an assigned weight whenever they share a connected component of $G$, then we can adjust our weights by connected component to allow for a compatible assignment of weight to $T_J$. This leaves us the case in which there are two subcomplexes $T_{J_1}$ and $T_{J_2}$ in the same connected component of $G$ which have been assigned different weights and both of which have an edge directed towards $T_J$.

Consider now this case. Consider a path $p$ from $T_{J_1}$ to $T_{J_2}$, possibly moving backwards along edges. We claim that such a path $p'$ exists which moves backwards along some non-negative number of edges and then moves forwards along some non-negative number of edges. To prove this, it suffices to show that, if at any point in $p$ we move forwards along one edge to some and then backwards along another, going from $g_1$ to $g_2$ to $g_3$ for some $g_1$, $g_2$, $g_3$, we could have instead moved backwards along an edge first and then forwards along another and arrived at the same destination. If our edges correspond to the elements $f_{\underline{i}}$ and $f_{\underline{j}}$ respectively, then for some $K\subset[n]\setminus\{i,j\}$ we have
\[g_1=\{j\}\sqcup K,\qquad g_2=\{i,j\}\sqcup K,\qquad g_3=\{i\}\sqcup K,\]
where $f_{\underline{i}}$ and $f_{\underline{j}}$ are coprime with $f_{\underline{j}K}$ and $f_{\underline{i}K}$ respectively. This implies that $f_{\underline{i}}$ and $f_{\underline{j}}$ are coprime with $f_{K}$, so we can replace $g_2$ with $K$ in $p$, so a path $p'$ as described above exists.

There must be some unique element $T_L$ along $p'$ with the highest assigned weight. Then we have paths moving forwards along edges from $T_L$ to both $T_{J_1}$ and $T_{J_2}$ in $G$. Our iterative assignment of weights to the vertices of $G$ and the fact that in this assignment $T_{J_1}$ and $T_{J_2}$ have different weights tell us that these paths have different lengths. By considering the additional edges from these to $T_J$, we have two paths of different lengths from $T_L$ to $T_J$ moving forward along edges. Say that these edges correspond to sets $S$ and $T$. Then we have
\[\operatorname{gcd}(f_L,f_S)=1,\quad\operatorname{gcd}(f_L,f_T)=1,\quad\text{and}\quad f_{L\cup S}=f_{L\cup T}.\]
The first two of these indicate that
\[f_{L\cup S}=f_{L}f_{S}\quad\text{and}\quad f_{L\cup T}=f_{L}f_{T},\]
yielding $f_S=f_T$, so $S$ and $T$ satisfy the conditions given in the theorem statement, contradiction.
\end{proof}
\noindent This theorem tells us that it is relatively easy to check whether a Taylor subcomplex admits a weak grading: a non-weakly-gradable Taylor subcomplex graph will have two paths from $T_{\varnothing}$ to some other vertex with different lengths, and furthermore these paths will each be given by coprime sets of elements of $\boldsymbol{\mathit{f}}$.

\section{Edge ideals of cycles}\label{edge}
The edge ideal of a $d$-cycle is $\boldsymbol{\mathit{f}}=x_1x_2,x_2x_3,\ldots,x_dx_1$ in $Q[x_1,\ldots,x_d]$. We begin with some preliminary information, then proceed with our two by-hand examples.
\subsection{Simplicial complexes concerning edge ideals}\label{inclusionsection}

We have a decent understanding of the simplicial complexes $\Delta_J$ corresponding to edge ideals thanks to \cite{kozlov}, whose results we will recall here in a more relevant form. Consider first $\Delta_{\underline{12\ldots i}}$ for some $2\leq i<n-1$. $\Delta_{\underline{12\ldots i}}$ has $i-2$ points, indexed from 2 to $i-1$, and consists of subsets of this set of indices containing no two adjacent values. 

\begin{lemma}[{\cite[Proposition 4.6]{kozlov}}]\label{godhom}
Say $i\geq1$ and $\boldsymbol{\mathit{f}}$ is the edge ideal of a cycle with at least $i+2$ generators. If $3\mid i$, then $\Delta_{\underline{1\ldots i}}$ has trivial reduced homology. Otherwise, its reduced homology is $\mathbb{Z}$ at entry $\left\lfloor\frac{n}{3}-1\right\rfloor$ and zero elsewhere.
\end{lemma}

Now consider some arbitrary $M_S$ which is not $M_{[n]}$. Note that $M_S$ can be considered as a union of sets of consecutive values (considering 1 and $4m+2$ consecutive) which is as small as possible. Then $\Delta_{M_S}$ will be isomorphic to the set union of the corresponding $\Delta_{\underline{1\ldots i}}$ complexes above, and will have homology their direct sum, whence their reduced homology will be able to be determined.

\begin{lemma}[{\cite[Proposition 5.2]{kozlov}}]\label{godhomtwo}
Say $n\geq2$ and $\boldsymbol{\mathit{f}}$ is the edge ideal of a cycle with $n$ generators. Then the reduced homology of $\Delta_{\underline{1\ldots n}}$ is zero except
\[\begin{cases}
\mathbb{Z}^2\text{ at entry }\frac{n}{3}-1&\text{when }n\equiv0\pmod3,\\
\mathbb{Z}\text{ at entry }\frac{n-1}{3}-1&\text{when }n\equiv1\pmod3,\\
\mathbb{Z}\text{ at entry }\frac{n-2}{3}&\text{when }n\equiv2\pmod3.
\end{cases}\]
\end{lemma}

In both of these cases, by the universal coefficient theorem, the reduced cohomologies and reduced homologies of these simplicial complexes will be isomorphic. Additionally, before continuing, it is worth noting that \cite{kozlov} contains information regarding the generating simplices of these simplicial complexes, which, should this path of inquiry prove fruitful, may be worth investigating further, especially given that we are working primarily with cohomology.

We also wish to elaborate a bit on the calculation and simplification of $\operatorname{ksgn}$. Note that $\operatorname{ksgn}(S)$ seeks the number of elements of $S$ which have even indices in $\operatorname{sort}(M_S\setminus S)$. We note that this is the sign of the permutation which takes the reversed $\operatorname{sort}(S)$, which we will write as $\operatorname{rsort}(S)$, and concatenates $\operatorname{sort}(M_S\setminus S)$ to it:
\[\operatorname{ksgn}(S)=\operatorname{sgn}(\operatorname{rsort}(S)\operatorname{sort}(M_S\setminus S)).\]
This will ease some of our sign-calculation pains moving forward.

Finally, as theorized in private communication with E.~Grifo, we have the following lemma which encapsulates a lower bound on our support varieties which will meaningfully simplify the remainder of our calculations:
\begin{lemma}
Any edge ideal of size $4m+2$ has in its support variety that given by $\mathcal{V}(a_{\underline{13\ldots(4m+1)}}+a_{\underline{24\ldots(4m+2)}})$. 
\end{lemma}
\begin{proof}For $i\in\mathbb{Z}$, let
\begin{itemize}
\item $S_i^+=\sigma^i\left(\{i\in[4m+2]\ |\ i\in1,2\pmod{4}\}\right)$,
\item $S_i^\ell=\sigma^i\left(S\setminus \{1\}\right)$,
\item $S_i^r=\sigma^i\left(S\setminus \{4m+2\}\right)$, and
\item $S_i^-=\sigma^i\left(S\setminus \{4m+1,4m+2\}\right)$.
\end{itemize}
Consider the subspace of $T^{\Sigma}(\boldsymbol{\mathit{f}},a)$ spanned by the set of basis elements
\[\bigcup\langle\sigma^2\rangle \{b_{\operatorname{sort}(S^\ell)},b_{\operatorname{sort}(S^r)}\}\]
where $\sigma$ is the permutation $(1\ 2 \ldots 4m+2)$. Its intersection with $d_aT^{\Sigma}(\boldsymbol{\mathit{f}},a)$ is zero, and its image under $d_a$ in $T^{\Sigma}(\boldsymbol{\mathit{f}},a)$ lies in the subspace with basis
\[\langle\sigma^2\rangle b_{S^+}\cup\langle\sigma^2\rangle b_{S^-}.\]
\noindent These two subspaces have dimension $4m+2$, so it suffices to show that the square matrix between these given by $d_a$ has zero determinant when $a_{\underline{13\ldots(4m+1)}}+a_{\underline{24\ldots(4m+2)}}=0$. We will let $a_{i(4m+2)+j}$ refer to $a_j$ for $T$ an ordered set and $i,j\in\mathbb{Z}$. We have
\[d_a\left(b_{S_i^-}\right)=a_{i-1}b_{S_i^r}+a_{i-2}b_{S_{i-4}^\ell}\qquad\text{and}\qquad d_a\left(b_{S_i^+}\right)=-b_{S_i^r}+b_{S_i^\ell},\]
so the map between these given by $d_a$ is given by the matrix
\[
\begin{blockarray}{cccccccc}
& b_{S^-} & b_{S^+} & b_{S_4^-} & b_{S_4^+} & b_{S_8^-} & & b_{S_{-4}^+} \\
\begin{block}{c[ccccccc]}
  b_{S^r} & a_{4m+1} & -1 & 0 & 0 & 0 & \cdots & 0 \\
  b_{S^\ell} & 0 & 1 & a_2 & 0 & 0 & \cdots & 0 \\
  b_{S_4^r} & 0 & 0 & a_3 & -1 & 0 & \cdots & 0 \\
  b_{S_4^\ell} & 0 & 0 & 0 & 1 & a_6 & \cdots & 0 \\
  b_{S_8^r} & 0 & 0 & 0 & 0 & a_7 & \cdots & 0 \\
   & \vdots & \vdots & \vdots & \vdots & \vdots & \ddots & \vdots \\
  b_{S_{-4}^\ell} & a_{4m} & 0 & 0 & 0 & 0 & 0 & 1 \\
\end{block}
\end{blockarray}
\]
which has determinant $a_{\underline{13\ldots4m+1}}+a_{\underline{24\ldots4m+2}}$, completing our proof.
\end{proof}

\subsection{The case \texorpdfstring{$\bm{d=6}$}{d=6}}
Let $d=6$, let $\boldsymbol{\mathit{f}}=x_{\underline{12}},x_{\underline{23}},\ldots,x_{\underline{61}}$, and let $\sigma$ be the permutation $(1\ 2\ldots6)$. Our sets $M_J$ are:
\begin{itemize}
\item $\varnothing$,
\item $\langle\sigma\rangle\underline{1}$, of which there are 6,
\item $\langle\sigma\rangle\underline{12}$, of which there are 6,
\item $\langle\sigma\rangle\underline{123}$, of which there are 6,
\item $\langle\sigma\rangle\underline{1234}$, of which there are 6,
\item $\underline{123456}$, of which there is 1,
\item $\langle\sigma^2\rangle \underline{14}$, of which there are 3.
\end{itemize}
A quotiented version of our weakly gradable Taylor subcomplex graph, which we have quotiented by a partition such that the graph induced by any two subsets of our partition is either empty or a unidirectional complete bipartite graph on the two subsets, is as follows:
\[
\begin{tikzcd}
	&& {\langle\sigma\rangle T_{\underline{123}}} \\
	{T_{\varnothing}} & {\langle\sigma\rangle T_{\underline{1}}} && {T_{\underline{123456}}} \\
	&& {\langle\sigma^2\rangle T_{\underline{14}}} \\
	{\langle\sigma\rangle T_{\underline{12}}} & {\langle\sigma\rangle T_{\underline{1234}}}
	\arrow[from=1-3, to=2-4]
	\arrow[from=2-1, to=2-2]
	\arrow[from=2-2, to=1-3]
	\arrow[from=2-2, to=3-3]
	\arrow[from=4-1, to=4-2]
\end{tikzcd}
\]
This graph has two connected components, which by Corollary \ref{components} we may treat separately. We consider first the component with five entries. We get the following double complex from Theorem \ref{doublecomplex}:
\[
\begin{tikzcd}
	{T_{\varnothing}^{\Sigma}} \\
	{\bigoplus\langle\sigma\rangle T_{\underline{1}}^{\Sigma}} \\
	{\bigoplus\langle\sigma\rangle T_{\underline{123}}^{\Sigma}\oplus \bigoplus\langle\sigma^2\rangle T_{\underline{14}}^{\Sigma}} \\
	{T_{\underline{123456}}^{\Sigma}}
	\arrow[from=1-1, to=2-1]
	\arrow[from=2-1, to=3-1]
	\arrow[from=3-1, to=4-1]
\end{tikzcd}
\]
Given this double complex, we can determine if its totalization has any homology using spectral sequences. Here we exploit the relationship between our $T_J^\Sigma$ complexes and reduced cellular cochain complexes. Let us first calculate our $\Delta_J$ so we can better understand their homology:
\begin{itemize}
\item $\Delta_\varnothing$, each of $\langle\sigma\rangle \Delta_{\underline{1}}$, and each of $\langle\sigma^2\rangle \Delta_{\underline{14}}$ are $\varnothing$ — no other $f_J$ are the same as those from these subsets of $\boldsymbol{\mathit{f}}$.
\item $\Delta_{\underline{123}}$ is the point 2, which we can write $\dotseg{2}$. $\langle\sigma\rangle$ can be applied to this to get the others.
\item $\Delta_{\underline{123456}}$ consists of all subsets of \underline{123456} without adjacent pairs (including the pair \underline{16}), which looks like the following:
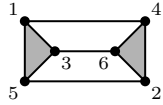
\[
\begin{tikzpicture}[scale=.4,line width=.75pt]
  \coordinate (1) at (0,1);
  \coordinate (5) at (0,-1);
  \coordinate (4) at (4,1);
  \coordinate (2) at (4,-1);
  \coordinate (3) at (1,0);
  \coordinate (6) at (3,0);

  \fill[color=black!30] (1)--(3)--(5)--cycle;
  \fill[color=black!30] (4)--(6)--(2)--cycle;

  \draw (1) -- (4) -- (2) -- (5) -- cycle;

  \draw (1) -- (3) -- (5);
  \draw (4) -- (6) -- (2);
  \draw (3) -- (6);

  \foreach \p/\pos in {1/above left, 5/below left, 4/above right, 2/below right, 3/below right, 6/below left}
    \node[circle,fill,scale=.4,label={[label distance=-4pt]\pos: {\scriptsize \p}}] at (\p) {};
\end{tikzpicture}
\]
\end{itemize}
In each of these cases, the reduced homology will be free, so our reduced cohomologies will be isomorphic by the universal coefficient theorem:
\begin{itemize}
\item $\varnothing$ has $\operatorname{H}^{-1}(\varnothing)\cong\mathbb{Z}$ and no other cohomology,
\item $\dotseg{2}$ has no reduced cohomology,
\item $\Delta_{\underline{123456}}$ has $\operatorname{H}^{1}(\Delta_{\underline{123456}})\cong\mathbb{Z}^2$ and no other cohomology.
\end{itemize}
We begin by taking horizontal derivatives. Letting $\operatorname{H}_J^i$ denote the reduced $i$-th cohomology of $\Delta_J$, our $E_1$ page will be isomorphic to the following, where our isomorphisms from $\operatorname{H}_J^{i-1}$ to $(\operatorname{H}(T_J^{\Sigma}))_{|M_J|-\Sigma_{T_J}-i}$ are inherited from the map taking $c_J'$ to $c_J^{\Sigma}$, where all non-zero entries are shown:
\[
\begin{tikzcd}
	0 & {\operatorname{H}_{\varnothing}^{-1}}\\
	0 & {\bigoplus\langle\sigma\rangle\operatorname{H}^{-1}_{\underline{1}}} \\
	0 & {\bigoplus\langle\sigma^2\rangle\operatorname{H}_{\underline{14}}^{-1}}\\
	{\operatorname{H}_{\underline{123456}}^1} & 0
	\arrow[from=1-1, to=1-2]
	\arrow[from=1-1, to=2-1]
	\arrow[from=1-2, to=2-2]
	\arrow[from=2-1, to=2-2]
	\arrow[from=2-1, to=3-1]
	\arrow[from=2-2, to=3-2]
	\arrow[from=3-1, to=3-2]
	\arrow[from=3-1, to=4-1]
	\arrow[from=3-2, to=4-2]
	\arrow[from=4-1, to=4-2]
\end{tikzcd}
\]
Next we construct $E_2$. We fix a basis for each of our non-zero terms in order to define maps between our entries as matrices:
\begin{itemize}
\item $\operatorname{H}_{\varnothing}^{-1}$ is generated by $c_\varnothing'$,
\item $\langle\sigma\rangle\operatorname{H}_{\underline{1}}^{-1}$ are generated by $\langle\sigma\rangle c_{\underline{1}}'$,
\item $\langle\sigma^2\rangle\operatorname{H}_{\underline{14}}$ are generated by $\langle\sigma^2\rangle c_{\underline{14}}'$,
\item $\operatorname{H}_{\underline{123456}}^1$ is generated by $c_{\underline{2356}}'$ and $-c_{\underline{1346}}'$.
\end{itemize}
\noindent Note that this includes generators such as $c_{\underline{52}}'=-c_{\underline{25}}'$. We can now define our $E_1$ maps by recalling that when non-zero, the value of $b_jc_J=e_jc_J$ is 
\[\operatorname{ksgn}(j,J)\operatorname{sgn}(\{j\}J)c_{\{j\}\cup J},\]
and our vertical maps are directly inherited from these multiplication maps:
\[\operatorname{H}_{\varnothing}^{-1}\to \bigoplus\langle\sigma\rangle\operatorname{H}_{\underline{1}}^{-1}:\begin{bmatrix}a_1\\a_2\\a_3\\a_4\\a_5\\a_6\end{bmatrix},\qquad\bigoplus\langle\sigma\rangle\operatorname{H}_{\underline{1}}^{-1}\to\bigoplus\langle\sigma\rangle\operatorname{H}_{\underline{14}}^{-1}:\begin{bmatrix}
a_4&0&0&-a_1&0&0\\
0&a_5&0&0&-a_2&0\\
0&0&a_6&0&0&-a_3
\end{bmatrix}.\]
The first of these maps is injective. The second map is surjective so long as no $i$ satisfies $a_i=a_{i+3}=0$. If this were the case, then $\langle\sigma^2\rangle\operatorname{H}_{\underline{14}}^{-1}$ will beget a non-zero $E_2$ entry, and consequently a non-zero $E_{\infty}$ entry. Thus we record that the points $a$ with $a_i=a_{i+3}=0$ for some $i$ lie in our support variety, and henceforth assume this is not the case. Our $E_2$ page is isomorphic to the following, where all non-zero entries are shown:
\[
\let\amsamp=&
\begin{tikzcd}
	0 & \begin{array}{c} \ker\begin{bmatrix} a_4\amsamp 0\amsamp 0\amsamp -a_1\amsamp 0\amsamp 0\\ 0\amsamp a_5\amsamp 0\amsamp 0\amsamp -a_2\amsamp 0\\ 0\amsamp 0\amsamp a_6\amsamp 0\amsamp 0\amsamp -a_3 \end{bmatrix}/\operatorname{im}\begin{bmatrix}a_1\\a_2\\a_3\\a_4\\a_5\\a_6\end{bmatrix} \end{array}\\
	0 & 0\\
	{\operatorname{H}_{\underline{123456}}^1} & 0
	\arrow[from=1-2, to=2-2]
	\arrow[from=2-2, to=3-2]
	\arrow[from=1-1, to=2-1]
	\arrow[from=2-1, to=3-1]
	\arrow[from=1-1, to=1-2]
	\arrow[from=2-1, to=2-2]
	\arrow[from=3-1, to=3-2]
\end{tikzcd}\]
The top-right entry is generated by $a_1c_{\underline{1}}^{\Sigma}+a_4c_{\underline{4}}^\Sigma$ and $a_2c_{\underline{2}}^{\Sigma}+a_5c_{\underline{5}}^{\Sigma}$, and thus are left only to determine the images of these in the entry isomorphic to $H_{\underline{123456}}^1$ under the $E_2$ map.

The images of these values under the vertical map are
\[-a_{\underline{13}}c_{\underline{13}}^{\Sigma}+a_{\underline{15}}c_{\underline{15}}^{\Sigma}+a_{\underline{24}}c_{\underline{24}}^{\Sigma}-a_{\underline{46}}c_{\underline{46}}^{\Sigma}\qquad\text{and}\qquad -a_{\underline{24}}c_{\underline{24}}^{\Sigma}+a_{\underline{26}}c_{\underline{26}}^{\Sigma}+a_{\underline{35}}c_{\underline{35}}^{\Sigma}-a_{\underline{15}}c_{\underline{15}}^{\Sigma}.\]
Noting that the $c_J'$ corresponding to these $c_J^{\Sigma}$ are all points, we can pull these back to achieve
\[-a_{\underline{13}}c_{\underline{123}}^{\Sigma}+a_{\underline{15}}c_{\underline{156}}^{\Sigma}+a_{\underline{24}}c_{\underline{234}}^{\Sigma}-a_{\underline{46}}c_{\underline{456}}^{\Sigma}\qquad\text{and}\qquad -a_{\underline{24}}c_{\underline{234}}^{\Sigma}+a_{\underline{26}}c_{\underline{126}}^{\Sigma}+a_{\underline{35}}c_{\underline{345}}^{\Sigma}-a_{\underline{15}}c_{\underline{156}}^{\Sigma}.\]
Finally, we take the image of these under our vertical map one more time, yielding
\[a_{\underline{135}}c_{\underline{1235}}^{\Sigma}-a_{\underline{135}}c_{\underline{1356}}^{\Sigma}-a_{\underline{246}}c_{\underline{2346}}^{\Sigma}-a_{\underline{246}}c_{\underline{2456}}^{\Sigma}\qquad\text{and}\qquad a_{\underline{246}}c_{\underline{2346}}^{\Sigma}+a_{\underline{246}}c_{\underline{1246}}^{\Sigma}+a_{\underline{135}}c_{\underline{1345}}^{\Sigma}+a_{\underline{135}}c_{\underline{1356}}^{\Sigma}.\]
Under our isomorphism with our reduced cellular cochain complex, we can write this as the following weighted sum of segments in $\Delta_{\underline{123456}}$:
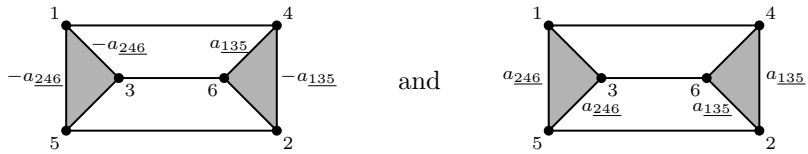
\[
\begin{tabular}{@{}c@{}}\begin{tikzpicture}[scale=.7,line width=.75pt]
  \coordinate (1) at (0,1);
  \coordinate (5) at (0,-1);
  \coordinate (4) at (4,1);
  \coordinate (2) at (4,-1);
  \coordinate (3) at (1,0);
  \coordinate (6) at (3,0);

  \fill[color=black!30] (1)--(3)--(5)--cycle;
  \fill[color=black!30] (4)--(6)--(2)--cycle;

  \draw (1) -- (4) -- (2) -- (5) -- cycle;

  \draw (1) -- (3) -- (5);
  \draw (4) -- (6) -- (2);
  \draw (3) -- (6);

  \foreach \p/\pos in {1/above left, 5/below left, 4/above right, 2/below right, 3/below right, 6/below left}
    \node[circle,fill,scale=.4,label={[label distance=-4pt]\pos: {\scriptsize \p}}] at (\p) {};
	\node at (3.1,.6) {\scriptsize $a_{\underline{135}}$};
	\node at (4.6,0) {\scriptsize $-a_{\underline{135}}$};
	\node at (-.6,0) {\scriptsize $-a_{\underline{246}}$};
	\node at (1,.6) {\scriptsize $-a_{\underline{246}}$};
\end{tikzpicture}\end{tabular}\qquad\text{and}\qquad\begin{tabular}{@{}c@{}}\begin{tikzpicture}[scale=.7,line width=.75pt]
  \coordinate (1) at (0,1);
  \coordinate (5) at (0,-1);
  \coordinate (4) at (4,1);
  \coordinate (2) at (4,-1);
  \coordinate (3) at (1,0);
  \coordinate (6) at (3,0);

  \fill[color=black!30] (1)--(3)--(5)--cycle;
  \fill[color=black!30] (4)--(6)--(2)--cycle;

  \draw (1) -- (4) -- (2) -- (5) -- cycle;

  \draw (1) -- (3) -- (5);
  \draw (4) -- (6) -- (2);
  \draw (3) -- (6);

  \foreach \p/\pos in {1/above left, 5/below left, 4/above right, 2/below right, 3/below right, 6/below left}
    \node[circle,fill,scale=.4,label={[label distance=-4pt]\pos: {\scriptsize \p}}] at (\p) {};
	\node at (3.1,-.6) {\scriptsize $a_{\underline{135}}$};
	\node at (4.5,0) {\scriptsize $a_{\underline{135}}$};
	\node at (-.5,0) {\scriptsize $a_{\underline{246}}$};
	\node at (1,-.6) {\scriptsize $a_{\underline{246}}$};
\end{tikzpicture}
\end{tabular}
\]
These are the following, up to a cocycle:
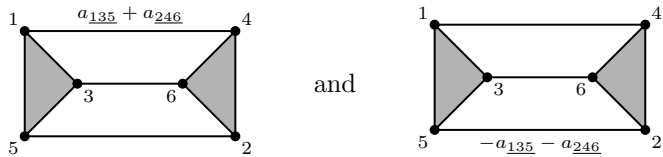
\[
\begin{tabular}{@{}c@{}}\begin{tikzpicture}[scale=.7,line width=.75pt]
  \coordinate (1) at (0,1);
  \coordinate (5) at (0,-1);
  \coordinate (4) at (4,1);
  \coordinate (2) at (4,-1);
  \coordinate (3) at (1,0);
  \coordinate (6) at (3,0);

  \fill[color=black!30] (1)--(3)--(5)--cycle;
  \fill[color=black!30] (4)--(6)--(2)--cycle;

  \draw (1) -- (4) -- (2) -- (5) -- cycle;

  \draw (1) -- (3) -- (5);
  \draw (4) -- (6) -- (2);
  \draw (3) -- (6);

  \foreach \p/\pos in {1/above left, 5/below left, 4/above right, 2/below right, 3/below right, 6/below left}
    \node[circle,fill,scale=.4,label={[label distance=-4pt]\pos: {\scriptsize \p}}] at (\p) {};
	\node at (2,1.3) {\scriptsize $a_{\underline{135}}+a_{\underline{246}}$};
\end{tikzpicture}\end{tabular}\qquad\text{and}\qquad\begin{tabular}{@{}c@{}}\begin{tikzpicture}[scale=.7,line width=.75pt]
  \coordinate (1) at (0,1);
  \coordinate (5) at (0,-1);
  \coordinate (4) at (4,1);
  \coordinate (2) at (4,-1);
  \coordinate (3) at (1,0);
  \coordinate (6) at (3,0);

  \fill[color=black!30] (1)--(3)--(5)--cycle;
  \fill[color=black!30] (4)--(6)--(2)--cycle;

  \draw (1) -- (4) -- (2) -- (5) -- cycle;

  \draw (1) -- (3) -- (5);
  \draw (4) -- (6) -- (2);
  \draw (3) -- (6);

  \foreach \p/\pos in {1/above left, 5/below left, 4/above right, 2/below right, 3/below right, 6/below left}
    \node[circle,fill,scale=.4,label={[label distance=-4pt]\pos: {\scriptsize \p}}] at (\p) {};
	\node at (2,-1.3) {\scriptsize $-a_{\underline{135}}-a_{\underline{246}}$};
\end{tikzpicture}
\end{tabular}
\]
Thus the image of the chosen basis of our source $E_2$ entry under this $E_2$ map are $a_{\underline{135}}+a_{\underline{246}}$ times the basis of our target $E_2$ entry. Thus the subcomplex of $T^\Sigma(\boldsymbol{\mathit{f}},a)$ corresponding to this component has trivial homology if and only if $a_{\underline{135}}+a_{\underline{246}}\neq0$. Note that this is true whenever $a_i=a_{i+3}=0$ for some $i$, so this subcomplex of $T^\Sigma(\boldsymbol{\mathit{f}},a)$ is exact exactly when $a_{\underline{135}}+a_{\underline{246}}\neq0$.

Now we must consider the other connected component of our given weakly gradable Taylor subcomplex graph. It yields the double complex
\[
\begin{tikzcd}
	{\bigoplus\langle\sigma\rangle T_{\underline{12}}^{\Sigma}} \\
	{\bigoplus\langle\sigma\rangle T_{\underline{1234}}^{\Sigma}}
	\arrow[from=1-1, to=2-1]
\end{tikzcd}
\]
Again we determine our $\Delta_J$ and their reduced cohomologies:
\begin{itemize}
\item Each of $\langle\sigma\rangle \Delta_{\underline{12}}$ is $\varnothing$, which has only $\operatorname{H}^{-1}=\mathbb{Z}$,
\item Each of $\langle\sigma\rangle \Delta_{\underline{1234}}$ are, respectively, $\langle\sigma\rangle\dotseg{2}\dotseg{3}$, which have $\operatorname{H}^0=\mathbb{Z}$ and no other cohomology.
\end{itemize}
Our $E_1$ page is isomorphic to the following:
\[
\begin{tikzcd}
	{\bigoplus\langle\sigma\rangle \operatorname{H}_{\underline{12}}^{-1}} \\
	{\bigoplus\langle\sigma\rangle \operatorname{H}_{\underline{1234}}^0}
	\arrow[from=1-1, to=2-1]
\end{tikzcd}
\]
We fix a basis for our non-zero terms:
\begin{itemize}
\item $\langle\sigma\rangle\operatorname{H}_{\underline{12}}^{-1}$ are generated respectively by $\langle\sigma\rangle c_{\underline{12}}'$,
\item $\langle\sigma\rangle\operatorname{H}_{\underline{1234}}^{-1}$ are generated respectively by $\langle\sigma\rangle c_{\underline{124}}'$.
\end{itemize}
Under these bases the vertical map is given by the matrix
\[
\begin{bmatrix}
-a_4&0&-a_1&0&0&0\\
0&-a_5&0&-a_2&0&0\\
0&0&-a_6&0&-a_3&0\\
0&0&0&a_1&0&a_4\\
-a_5&0&0&0&-a_2&0\\
0&a_6&0&0&0&a_3
\end{bmatrix}\,,
\]
which up to a change of basis is
\[
\begin{tabular}{@{}c@{}}$\begin{bmatrix}
a_4&a_1&0\\
0&a_6&a_3\\
a_5&0&a_2
\end{bmatrix}$\end{tabular}\oplus\begin{tabular}{@{}c@{}}$\begin{bmatrix}
a_5&a_2&0\\
0&a_1&a_4\\
a_6&0&a_3
\end{bmatrix}$\end{tabular}\,,
\]
both summands of which have determinant $a_{\underline{135}}+a_{\underline{246}}$. We have successfully proven that the cohomological support variety of this edge ideal is $\mathcal{V}(a_{\underline{135}}+a_{\underline{246}})$.

\subsection{The case \texorpdfstring{$\bm{d=10}$}{d=10}}

Let $d=10$ and let $\boldsymbol{\textit{f}}=(x_{\underline{12}},\ldots,x_{\underline{A1}})$. Furthermore, let $k$ have characterstic not equal to 2 or 5. We again know that our Taylor subcomplex graph will have two components, corresponding to the parity of the degrees of the LCMs of the sets of elements. We will start with the component yielding odd degrees. Our sets $M_J$ are as follows, where in parentheses we have given our collections of sets names based on the lengths of sets of consecutive variables in their corresponding $f_J$ values, where $\sigma$ is the permutation $(1\ 2\ldots10)$:
\begin{itemize}
\item $\langle\sigma\rangle\{\underline{12}\}$, of which there are 10 ($S(3)$),
\item $\langle\sigma\rangle\{\underline{125},\underline{126},\underline{127},\underline{128}\}$, of which there are 40 ($S(3,2)$),
\item $\bigcup\langle\sigma\rangle\{\underline{1234}\}$, of which there are 10 ($S(5)$),
\item $\bigcup\langle\sigma\rangle\{\underline{1258}\}$, of which there are 10 ($S(3,2,2)$),
\item $\bigcup\langle\sigma\rangle\{\underline{12567},\underline{12678}\}$, of which there are 20 ($S(3,4)$),
\item $\bigcup\langle\sigma\rangle\{\underline{12347},\underline{12348}\}$, of which there are 20 ($S(5,2)$),
\item $\bigcup\langle\sigma\rangle\{\underline{123456}\}$, of which there are 10 ($S(7)$), and
\item $\bigcup\langle\sigma\rangle\{\underline{123456789}\}$, of which there are 10 ($S(9)$).
\end{itemize}
Our weakly gradable Taylor subcomplex graph quotiented by the above partition is as follows:
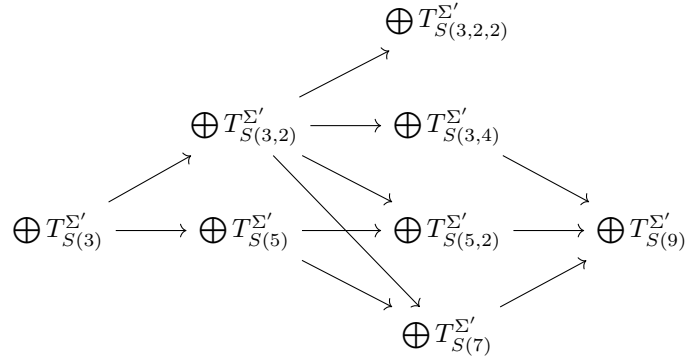
\[\begin{tikzcd}
	&& {T_{S(3,2,2)}^{\Sigma}} \\
	& {T_{S(3,2)}^{\Sigma}} & {T_{S(3,4)}^{\Sigma}} \\
	{T_{S(3)}^{\Sigma}} & {T_{S(5)}^{\Sigma}} & {T_{S(5,2)}^{\Sigma}} & {T_{S(9)}^{\Sigma}} \\
	&& {T_{S(7)}^{\Sigma}}
	\arrow[from=2-2, to=1-3]
	\arrow[from=2-2, to=2-3]
	\arrow[from=2-2, to=3-3]
	\arrow[from=2-2, to=4-3]
	\arrow[from=2-3, to=3-4]
	\arrow[from=3-1, to=2-2]
	\arrow[from=3-1, to=3-2]
	\arrow[from=3-2, to=3-3]
	\arrow[from=3-2, to=4-3]
	\arrow[from=3-3, to=3-4]
	\arrow[from=4-3, to=3-4]
\end{tikzcd}\]
\noindent We then have the following double complex:
\[\begin{tikzcd}
	{\bigoplus T_{S(3)}^{\Sigma}} \\
	{\bigoplus T_{S(3,2)}^{\Sigma}\oplus\bigoplus T_{S(5)}^{\Sigma}} \\
	{\bigoplus T_{S(3,2,2)}^{\Sigma}\oplus\bigoplus T_{S(3,4)}^{\Sigma}\oplus\bigoplus T_{S(5,2)}^{\Sigma}\oplus\bigoplus T_{S(7)}^{\Sigma}} \\
	{\bigoplus T_{S(9)}^{\Sigma}}
	\arrow[from=1-1, to=2-1]
	\arrow[from=2-1, to=3-1]
	\arrow[from=3-1, to=4-1]
\end{tikzcd}\]

Now we take the horizontal homology using Lemma \ref{godhom}, yielding the following diagram, where all of the non-zero entries are shown:
\[\begin{tikzcd}
	0 & {\bigoplus\operatorname{H}_{S(3)}^{-1}} \\
	0 & \begin{array}{c}\bigoplus\operatorname{H}_{S(3,2)}^{-1}\\\oplus\bigoplus\operatorname{H}_{S(5)}^{0}\end{array} \\
0 & \begin{array}{c}\bigoplus\operatorname{H}_{S(3,2,2)}^{-1}\\\oplus\bigoplus\operatorname{H}_{S(5,2)}^0\end{array} \\
	{\bigoplus\operatorname{H}_{S(9)}^1} & 0
	\arrow[from=1-1, to=1-2]
	\arrow[from=1-1, to=2-1]
	\arrow[from=1-2, to=2-2]
	\arrow[from=2-1, to=2-2]
	\arrow[from=2-1, to=3-1]
	\arrow[from=2-2, to=3-2]
	\arrow[from=3-1, to=3-2]
	\arrow[from=3-1, to=4-1]
	\arrow[from=3-2, to=4-2]
	\arrow[from=4-1, to=4-2]
\end{tikzcd}\]

Now in the right-hand column, since all of our simplicial complexes are quite simple, we can see that the first non-trivial map is injective and the second is surjective, so the homology at the middle non-trivial entry in that column has rank 10. It suffices to manually find distinct entries here which we can push down, then pull back, then push down, and show that the results yield non-trivial cocycles in $\bigoplus\operatorname{H}_{S(9)}^1$. These will prove the rank of our only potentially non-trivial $E_2$ map and consequently prove the exactness of our double complex. Consider 
\begin{align*}
a_{\underline{125}}c_{\underline{125}}+a_{\underline{128}}c_{\underline{128}}-a_{\underline{248}}c_{\underline{348}}-a_{\underline{159}}c_{\underline{59A}}
\end{align*}
Pushing this down yields
\begin{align*}
&-a_{\underline{1257}}c_{\underline{1257}}+a_{\underline{1258}}c_{\underline{1258}}+a_{\underline{1259}}c_{\underline{1259}}\\
&\qquad-a_{\underline{1248}}c_{\underline{1248}}-a_{\underline{1258}}c_{\underline{1258}}+a_{\underline{1268}}c_{\underline{1268}}\\
&\qquad-a_{\underline{1248}}c_{\underline{1348}}-a_{\underline{2468}}c_{\underline{3468}}+a_{\underline{248A}}c_{\underline{348A}}\\
&\qquad+a_{\underline{1259}}c_{\underline{259A}}-a_{\underline{1359}}c_{\underline{359A}}+a_{\underline{1579}}c_{\underline{579A}}\\\\
=\,&-a_{\underline{1257}}c_{\underline{1257}}+a_{\underline{1259}}(c_{\underline{1259}}+c_{\underline{259A}})\\
&\qquad-a_{\underline{1248}}(c_{\underline{1248}}+c_{\underline{1348}})+a_{\underline{1268}}c_{\underline{1268}}\\
&\qquad-a_{\underline{2468}}c_{\underline{3468}}+a_{\underline{248A}}c_{\underline{348A}}\\
&\qquad-a_{\underline{1359}}c_{\underline{359A}}+a_{\underline{1579}}c_{\underline{579A}}.
\end{align*}
Pulling this back, which can be done manually using simplicial complexes by considering the pullbacks of the corresponding $c'$ terms, yields
\begin{align*}
&-a_{\underline{1257}}c_{\underline{12567}}+a_{\underline{1259}}c_{\underline{1259A}}\\
&\qquad-a_{\underline{1248}}c_{\underline{12348}}+a_{\underline{1268}}c_{\underline{12678}}\\
&\qquad+a_{\underline{2468}}c_{\underline{34678}}+a_{\underline{248A}}c_{\underline{3489A}}\\
&\qquad-a_{\underline{1359}}c_{\underline{3459A}}+a_{\underline{1579}}c_{\underline{5679A}},
\end{align*}
and pushing it down once more yields
\begin{align*}
&a_{\underline{12579}}c_{\underline{125679}}-a_{\underline{12579}}c_{\underline{12579A}}\\
&\qquad-a_{\underline{12468}}c_{\underline{123468}}+a_{\underline{12468}}c_{\underline{124678}}\\
&\qquad+(a_{\underline{12468}}c_{\underline{134678}}+a_{\underline{2468A}}c_{\underline{34678A}})+a_{\underline{2468A}}c_{\underline{34689A}}\\
&\qquad+a_{\underline{13579}}c_{\underline{34579A}}+(-a_{\underline{12579}}c_{\underline{25679A}}+a_{\underline{13579}}c_{\underline{35679A}})\\\\
&a_{\underline{12579}}(c_{\underline{125679}}-c_{\underline{12579A}}-c_{\underline{25679A}})\\
&\qquad+a_{\underline{12468}}(-c_{\underline{123468}}+c_{\underline{124678}}+c_{\underline{134678}})\\
&\qquad+a_{\underline{2468A}}c_{\underline{34678A}}+a_{\underline{2468A}}c_{\underline{34689A}}\\
&\qquad+a_{\underline{13579}}c_{\underline{34579A}}+a_{\underline{13579}}c_{\underline{35679A}}
\end{align*}
By considering simplicial complexes, the first two lines clearly comprise trivial cocycles, whereas the last comprises $a_{\underline{13579}}+a_{\underline{2468A}}$ times a non-trivial cocycle, proving that this subcomplex is exact as we have assumed that this value is non-zero.

Now we must consider the even-degree LCMs. We have the following $M_J$, similarly partitioned:
\begin{itemize}
\item $\varnothing$ (1),
\item $\langle\sigma\rangle\underline{1}$ ($S(2)$, 10),
\item $\langle\sigma\rangle\{\underline{14},\underline{15},\underline{16}\}$ ($S(2,2)$, 25),
\item $\langle\sigma\rangle\underline{123}$ ($S(4)$, 10),
\item $\langle\sigma\rangle\underline{147}$ ($S(2,2,2)$, 10),
\item $\langle\sigma\rangle\{\underline{1456},\underline{1567},\underline{1678}\}$ ($S(2,4)$, 30),
\item $\langle\sigma\rangle\{\underline{1256},\underline{1267}\}$ ($S(3,3)$, 15),
\item $\langle\sigma\rangle\underline{12345}$ ($S(6)$, 10),
\item $\langle\sigma\rangle\underline{145678}$ ($S(2,6)$, 10),
\item $\langle\sigma\rangle\underline{125678}$ ($S(3,5)$, 10),
\item $\langle\sigma\rangle\underline{123678}$ ($S(4,4)$, 5),
\item $\langle\sigma\rangle\underline{1234567}$ ($S(8)$, 10),
\item $\langle\sigma\rangle\underline{123456789A}$ ($S(10)$, 1).
\end{itemize}
This gives us the following:

\[\begin{tikzcd}
	{T_{\varnothing}^{\Sigma}} \\
	{\bigoplus T_{S(2)}^{\Sigma}} \\
	{\bigoplus T_{S(2,2)}^{\Sigma}\oplus\bigoplus T_{S(4)}^{\Sigma}} \\
	{\bigoplus T_{S(2,2,2)}^{\Sigma}\oplus\bigoplus T_{S(2,4)}^{\Sigma}\oplus\bigoplus T_{S(3,3)}^{\Sigma}\oplus\bigoplus T_{S(6)}^{\Sigma}} \\
	{\bigoplus T_{S(2,6)}^{\Sigma}\oplus\bigoplus T_{S(3,5)}^{\Sigma}\oplus\bigoplus T_{S(4,4)}^{\Sigma}\oplus\bigoplus T_{S(8)}^{\Sigma}} \\
	{T_{\underline{123456789A}}^{\Sigma}}
	\arrow[from=1-1, to=2-1]
	\arrow[from=2-1, to=3-1]
	\arrow[from=3-1, to=4-1]
	\arrow[from=4-1, to=5-1]
	\arrow[from=5-1, to=6-1]
\end{tikzcd}\]
Now we take the horizontal homology using Lemmas \ref{godhom} and \ref{godhomtwo}, where all of the non-zero entries are shown:

\[\begin{tikzcd}
	0 & 0 & {\operatorname{H}_{\varnothing}^{-1}} \\
	0 & 0 & {\bigoplus\operatorname{H}^{-1}_{S(2)}} \\
	0 & 0 & {\bigoplus\operatorname{H}^{-1}_{S(2,2)}} \\
	0 & \begin{array}{c} \bigoplus\operatorname{H}^{-1}_{S(3,3)}\\\oplus\bigoplus\operatorname{H}^{0}_{S(6)} \end{array} & {\bigoplus\operatorname{H}^{-1}_{S(2,2,2)}} \\
	0 & \begin{array}{c} \bigoplus\operatorname{H}^{0}_{S(2,6)}\\\oplus\bigoplus\operatorname{H}^{0}_{S(3,5)}\\\oplus\bigoplus\operatorname{H}^{1}_{S(8)} \end{array} & 0 \\
	{\operatorname{H}^{2}_{\underline{123456789A}}} & 0 & 0
	\arrow[from=1-1, to=1-2]
	\arrow[from=1-1, to=2-1]
	\arrow[from=1-2, to=1-3]
	\arrow[from=1-2, to=2-2]
	\arrow[from=1-3, to=2-3]
	\arrow[from=2-1, to=2-2]
	\arrow[from=2-1, to=3-1]
	\arrow[from=2-2, to=2-3]
	\arrow[from=2-2, to=3-2]
	\arrow[from=2-3, to=3-3]
	\arrow[from=3-1, to=3-2]
	\arrow[from=3-1, to=4-1]
	\arrow[from=3-2, to=3-3]
	\arrow[from=3-2, to=4-2]
	\arrow[from=3-3, to=4-3]
	\arrow[from=4-1, to=4-2]
	\arrow[from=4-1, to=5-1]
	\arrow[from=4-2, to=4-3]
	\arrow[from=4-2, to=5-2]
	\arrow[from=4-3, to=5-3]
	\arrow[from=5-1, to=5-2]
	\arrow[from=5-1, to=6-1]
	\arrow[from=5-2, to=5-3]
	\arrow[from=5-2, to=6-2]
	\arrow[from=5-3, to=6-3]
	\arrow[from=6-1, to=6-2]
	\arrow[from=6-2, to=6-3]
\end{tikzcd}\]
We can calculate the ranks of these maps. The top-right vertical map is injective. The map below can be seen to require the ratio of the above map in a source element in order to have image zero, so it has rank 9. The following map is surjective, as seen by considering images of $\langle\sigma\rangle c_{\underline{15}}$, so it has rank 10. The first non-zero vertical map in the middle column can be considered as follows. The images of $\bigoplus\operatorname{H}^0_{S(6)}$ biject via restriction to $\operatorname{H}^0_{S(2,6)}$. The images of $\bigoplus\operatorname{H}^{-1}_{\langle\sigma\rangle\{1267\}}$ inject via restriction to $\bigoplus\operatorname{H}^1_{S(8)}$. Thus it suffices to consider the map from $\bigoplus\operatorname{H}^{-1}_{\langle\sigma\rangle\{1256\}}$ to $\bigoplus\operatorname{H}^1_{S(3,5)}$. Since we may assume that $a_{\underline{13579}}+a_{\underline{2468A}}\neq0$, this map has full rank. To summarize, we have the following ranks of maps on our $E_1$ page:
\[\begin{tikzcd}
	0 & 0 & {\operatorname{H}_{\varnothing}^{-1}} \\
	0 & 0 & {\bigoplus\operatorname{H}^{-1}_{S(2)}} \\
	0 & 0 & {\bigoplus\operatorname{H}^{-1}_{S(2,2)}} \\
	0 & \begin{array}{c} \bigoplus\operatorname{H}^{-1}_{S(3,3)}\\\oplus\bigoplus\operatorname{H}^{0}_{S(6)} \end{array} & {\bigoplus\operatorname{H}^{-1}_{S(2,2,2)}} \\
	0 & \begin{array}{c} \bigoplus\operatorname{H}^{0}_{S(2,6)}\\\oplus\bigoplus\operatorname{H}^{0}_{S(3,5)}\\\oplus\bigoplus\operatorname{H}^{1}_{S(8)} \end{array} & 0 \\
	{\operatorname{H}^{2}_{\underline{123456789A}}} & 0 & 0
	\arrow[from=1-1, to=1-2]
	\arrow[from=1-1, to=2-1]
	\arrow[from=1-2, to=1-3]
	\arrow[from=1-2, to=2-2]
	\arrow["1", from=1-3, to=2-3]
	\arrow[from=2-1, to=2-2]
	\arrow[from=2-1, to=3-1]
	\arrow[from=2-2, to=2-3]
	\arrow[from=2-2, to=3-2]
	\arrow["9", from=2-3, to=3-3]
	\arrow[from=3-1, to=3-2]
	\arrow[from=3-1, to=4-1]
	\arrow[from=3-2, to=3-3]
	\arrow[from=3-2, to=4-2]
	\arrow["10", from=3-3, to=4-3]
	\arrow[from=4-1, to=4-2]
	\arrow[from=4-1, to=5-1]
	\arrow[from=4-2, to=4-3]
	\arrow["25", from=4-2, to=5-2]
	\arrow[from=4-3, to=5-3]
	\arrow[from=5-1, to=5-2]
	\arrow[from=5-1, to=6-1]
	\arrow[from=5-2, to=5-3]
	\arrow[from=5-2, to=6-2]
	\arrow[from=5-3, to=6-3]
	\arrow[from=6-1, to=6-2]
	\arrow[from=6-2, to=6-3]
\end{tikzcd}\]
Thus we have the following dimensions on our $E_2$ page:
\[\begin{tikzcd}
	0 & 0 & 0 \\
	0 & 0 & 0 \\
	0 & 0 & 6 \\
	0 & 0 & 0 \\
	0 & 5 & 0 \\
	1 & 0 & 0
	\arrow[from=1-1, to=1-2]
	\arrow[from=1-1, to=2-1]
	\arrow[from=1-2, to=1-3]
	\arrow[from=1-2, to=2-2]
	\arrow[from=1-3, to=2-3]
	\arrow[from=2-1, to=2-2]
	\arrow[from=2-1, to=3-1]
	\arrow[from=2-2, to=2-3]
	\arrow[from=2-2, to=3-2]
	\arrow[from=2-3, to=3-3]
	\arrow[from=3-1, to=3-2]
	\arrow[from=3-1, to=4-1]
	\arrow[from=3-2, to=3-3]
	\arrow[from=3-2, to=4-2]
	\arrow[from=3-3, to=4-3]
	\arrow[from=4-1, to=4-2]
	\arrow[from=4-1, to=5-1]
	\arrow[from=4-2, to=4-3]
	\arrow[from=4-2, to=5-2]
	\arrow[from=4-3, to=5-3]
	\arrow[from=5-1, to=5-2]
	\arrow[from=5-1, to=6-1]
	\arrow[from=5-2, to=5-3]
	\arrow[from=5-2, to=6-2]
	\arrow[from=5-3, to=6-3]
	\arrow[from=6-1, to=6-2]
	\arrow[from=6-2, to=6-3]
\end{tikzcd}\]

The subquotient of $\bigoplus\operatorname{H}^1_{S(2,2)}$ in our $E_2$ page can be generated by $\langle\sigma\rangle c_{\underline{16}}$ and $\sum\langle\sigma\rangle a_{\underline{15}}c_{\underline{15}}-\sum\langle\sigma^2\rangle a_{\underline{14}}c_{\underline{14}}$. It suffices to show that the first five of these can be taken by pushing and pulling onto a basis of our rank-5 entry in the middle column, and that the last of these can be taken by pushing and pulling onto some non-zero entry in the first column.

Let us first consider the image of $c_{\underline{16}}$. This will, up to a sign, by symmetry, give us the images of $\langle\sigma\rangle c_{\underline{16}}$. We are fortunate that signs need not be tracked in this calculation. Our first push down yields
\begin{align*}
&a_3b_{3}c_{\underline{16}}+a_4b_{4}c_{\underline{16}}a_8+b_{8}c_{\underline{16}}+a_9b_{9}c_{\underline{16}}\\
=\,&\pm a_3c_{\underline{136}}\pm a_4c_{\underline{146}}\pm a_8c_{\underline{168}}\pm a_9c_{\underline{169}}.
\end{align*}
\noindent To pull these left, it is easier to consider simplicial complexes. Each of these $c_S$ terms has a corresponding $c_S'$ in a simplicial complex consisting of a single point, thus they can be pulled left by removing the point, so pulling left yields
\[\pm a_3c_{\underline{1236}}\pm a_4c_{\underline{1456}}\pm a_8c_{\underline{1678}}\pm a_9c_{\underline{169A}}.\]
\noindent To push down for the second time, we will record only the terms coming from $S(2,6)$, which will be sufficient to determine that we have found a non-trivial cocycle. These are
\begin{align*}
&\pm a_{\underline{39}}b_9c_{\underline{1236}}\pm a_{\underline{48}}b_8c_{\underline{1456}}\pm a_{\underline{48}}b_4c_{\underline{1678}}\pm a_{\underline{39}}b_3c_{\underline{169A}}\\
=\,&\pm a_{\underline{39}}c_{\underline{12369}}\pm a_{\underline{48}}c_{\underline{14568}}\pm a_{\underline{48}}c_{\underline{14678}}\pm a_{\underline{39}}c_{\underline{1369A}}.
\end{align*}
Now it suffices to show that any cocycle with these coefficients from $S(2,6)$ is non-trivial in $E_2$, that is, that it can't be pulled back. Again, this is clear by considering simplicial complexes. Unless $a_{\underline{48}}=a_{\underline{39}}=0$, contradicting our assumption that $a_{\underline{13579}}+a_{\underline{2468A}}\neq0$, our corresponding $c_S'$ terms consist of some non-zero formal sum of points among $\Delta_{\underline{12369A}}$ and $\Delta_{\underline{145678}}$, which is a sum of either one or two non-zero scalar multiples of points on at least one of these. Since both of these simplicial complexes have 3 points, this can not be pulled back. By the same logic, no formal sum of elements of $\langle\sigma\rangle c_{\underline{16}}$ can be pulled back along this arrow. Thus this space bijects onto its target space in $E_2$.

This leaves us to prove that our remaining basis vector of this space in $E_2$, $\sum\langle\sigma\rangle a_{\underline{15}}c_{\underline{15}}-\sum\langle\sigma^2\rangle a_{\underline{14}}c_{\underline{14}}$, has a nontrivial image in $E_3$. We push down for the first time:
\begin{align*}
&\sum\langle\sigma\rangle \left(a_{\underline{715}}b_7c_{\underline{15}}+a_{\underline{815}}b_8c_{\underline{15}}+a_{\underline{915}}b_9c_{\underline{15}}\right)\\
&\qquad-\sum\langle\sigma^2\rangle\left(a_{\underline{714}}b_7c_{\underline{14}}+a_{\underline{814}}b_8c_{\underline{14}}\right)\\
=\,&\sum\langle\sigma\rangle\left(\operatorname{ksgn}(7,\underline{15})a_{\underline{715}}c_{\underline{715}}+\operatorname{ksgn}(8,\underline{15})a_{\underline{815}}c_{\underline{815}}+\operatorname{ksgn}(9,\underline{15})a_{\underline{915}}c_{\underline{915}}\right)\\
&\qquad-\sum\langle\sigma^2\rangle\left(\operatorname{ksgn}(7,\underline{14})a_{\underline{714}}c_{\underline{714}}+\operatorname{ksgn}(8,\underline{14})a_{\underline{814}}c_{\underline{814}}\right)\\
=\,&\sum\langle\sigma\rangle\left(\operatorname{ksgn}(7,\underline{15})a_{\underline{715}}c_{\underline{715}}+\operatorname{ksgn}(8,\underline{15})a_{\underline{815}}c_{\underline{815}}+\operatorname{ksgn}(9,\underline{15})a_{\underline{915}}c_{\underline{915}}\right)\\
&\qquad-\sum\langle\sigma\rangle\left(\operatorname{ksgn}(1,\underline{58})a_{\underline{815}}c_{\underline{815}}\right)
\end{align*}
Noting that $\sigma^i\operatorname{ksgn}(8,\underline{15})=\operatorname{ksgn}(1,\underline{58})=1$ for all $i$, we are left with
\begin{align*}
&\sum\langle\sigma\rangle\operatorname{ksgn}(7,\underline{15})a_{\underline{715}}c_{\underline{715}}+\operatorname{ksgn}(9,\underline{15})a_{\underline{915}}c_{\underline{915}}.
\end{align*}
\noindent We pull back for the first time using simplicial complexes:
\begin{align*}
&\sum\langle\sigma\rangle\operatorname{sgn}(7156)a_{\underline{157}}c_{\operatorname{sort}(\underline{1567})}+\operatorname{sgn}(915A)a_{\underline{159}}c_{\operatorname{sort}(\underline{159A})}.
\end{align*}
We replace our second term with a shift:
\begin{align*}
&\sum\langle\sigma\rangle\operatorname{sgn}(7156)a_{\underline{157}}c_{\operatorname{sort}(\underline{1567})}+\operatorname{sgn}(5716)a_{\underline{715}}c_{\operatorname{sort}(\underline{1567})}\\
=\,&2\sum\langle\sigma\rangle \operatorname{sgn}(7156)a_{\underline{157}}c_{\operatorname{sort}(\underline{1567})}
\end{align*}
Note that $\sigma^i\operatorname{sgn}(\underline{715})=1$ for all $i$.
We push down for the second time:
\begin{align*}
&2\sum\langle\sigma\rangle\operatorname{sgn}(37156\operatorname{sort}(24))a_{\underline{1357}}c_{\operatorname{sort}(13567)}\\
&\qquad+\operatorname{sgn}(97156\operatorname{sort}(8A))a_{\underline{1579}}c_{\operatorname{sort}(15679)}
\end{align*}
We can shift the second term for future simplicity:
\begin{align*}
&2\sum\langle\sigma\rangle\operatorname{sgn}(37156\operatorname{sort}(24))a_{\underline{1357}}c_{\operatorname{sort}(13567)}\\
&\qquad+\operatorname{sgn}(53712\operatorname{sort}(46))a_{\underline{1357}}c_{\operatorname{sort}(12357)}.
\end{align*}
We pull back for the second time:
\begin{align*}
&2\sum\langle\sigma\rangle\operatorname{sgn}(3715624)a_{\underline{1357}}\left(c_{\operatorname{sort}(\underline{123567})}-c_{\operatorname{sort}(\underline{134567})}\right)\\
&\qquad+\operatorname{sgn}(5371246)a_{\underline{1357}}\left(c_{\operatorname{sort}(\underline{123457})}-c_{\operatorname{sort}(\underline{123567})}\right)\\
=\,&2\sum\langle\sigma\rangle\operatorname{sgn}(3715624)a_{\underline{1357}}\left(2c_{\operatorname{sort}(\underline{123567})}-c_{\operatorname{sort}(\underline{134567})}-c_{\operatorname{sort}(\underline{123457})}\right).
\end{align*}
We push down for the third and final time:
\begin{align*}
&\sum\langle\sigma\rangle a_{\underline{13579}}\big(\operatorname{sgn}(9371562)2c_{\operatorname{sort}(\underline{1235679})}\\
&\qquad\qquad\qquad+\operatorname{sgn}(9371564)c_{\operatorname{sort}(\underline{1345679})}\\
&\qquad\qquad\qquad+\operatorname{sgn}(9371542)c_{\operatorname{sort}(\underline{1234579})}\big)\\
=\,&\sum\langle\sigma\rangle a_{\underline{13579}}\big(2c_{\operatorname{sort}(\underline{1235679})}\\
&\qquad\qquad\qquad-c_{\operatorname{sort}(\underline{1345679})}-c_{\operatorname{sort}(\underline{1234579})}\big).
\end{align*}
We can shift the last term:
\begin{align*}
&2\sum\langle\sigma\rangle a_{\underline{13579}}\left(c_{\operatorname{sort}(\underline{1235679})}-c_{\operatorname{sort}(\underline{1345679})}\right).
\end{align*}
\noindent Now consider $\Delta_{\underline{1\ldots A}}$. We wish to show that here,
\begin{align*}
&2\sum\langle\sigma\rangle a_{\underline{13579}}\left(c_{\operatorname{sort}(\underline{1235679})}-c_{\operatorname{sort}(\underline{1345679})}\right)
\end{align*}
has non-trivial cohomology. Let $v_{\operatorname{sort}(S)}$ be the coefficent of the simplex corresponding to $S$ in $\Delta_{1\ldots A}$ in a proposed preimage of the above cocycle and let $v_{S}=\operatorname{sgn}(S)v_{\operatorname{sort}(S)}$. Note that this cocycle has dot product $10(a_{\underline{13579}}+a_{\underline{2468A}})$ with the following formal sum of simplices:
\[\sum\langle\sigma\rangle\left(\underline{246}-2\cdot\underline{247}+\underline{258}\right).\]
\noindent This yields the following conditions imposed by our $v_S$ terms:
\[10(a_{\underline{13579}}+a_{\underline{2468A}})=\sum\langle\sigma\rangle \left(\left(2v_{\underline{24}}-v_{\underline{26}}\right)-2\left(v_{\underline{24}}-2v_{\underline{25}}\right)+2\left(v_{\underline{25}}+v_{\underline{26}}\right)\right)=0,\]
\noindent contradiction.

\section{Computational implementations}\label{compimp}
We have posted to \url{https://github.com/mgintz289/levels} a fork of the Macaulay2 package \sloppy\texttt{ThickSubcategories.m2} \cite{ThickSubcategoriesSource}. It introduces the method \texttt{equigeneratedMonomialCSV} and a test file \texttt{test-gintz-thesis.m2}.  The method
\begin{itemize}
\item takes an equigenerated monomial ideal $\boldsymbol{\mathit{f}}$ as input,
\item decomposes the double complex described by partitioning its basis by the degrees of the LCMs of its entries modulo the degree of the entries of $\boldsymbol{\mathit{f}}$,
\item computes the support of the homology of each component, and
\item returns their union.
\end{itemize}

\noindent This code builds our double complex using the values of $\Sigma_J$ using Lemma \ref{theoremareal}. The integration of our $\Sigma_J$ terms into this workflow should make it easy to rework this code to generate cohomological support varieties for other monomial ideals which can be constructed using this framework.

The test file computationally verifies Computations \ref{code1} and \ref{code2}:
\begin{computation}\label{code1}
The edge ideal of a 14-cycle over $\mathbb{Q}$ has support variety
\[\mathcal{V}(a_1a_3\cdots a_{13}+a_2a_4\cdots a_{14}).\]
\end{computation}

\begin{computation}\label{code2}
The cohomological support varieties of equigenerated monomial ideals in polynomial rings over $\mathbb{Q}$ with minimal generating sets with 6 elements are all one of the following up to order:
\begin{itemize}
\item a linear subspace,
\item a union of two hyperplanes,
\item $\mathcal{V}(a_{\underline{135}}+a_{\underline{246}})$.
\end{itemize}
\end{computation}
\noindent Of course, we have made no statements regarding non-equigenerated monomial ideals or other base fields, so we have not proven the original problem posed in \cite{embdef}.

This code constructs a collection of equigenerated monomial ideals such that every cohomological support variety of such an ideal, provided its minimal generating set has 6 generators, is represented, possibly excluding $\{0\}$ (though, in practice, this is not the case). Our construction revolves around the following
\begin{lemma}\label{codenameprin}
The cohomological support variety of a monomial ideal can be determined by whether, for any given $J\subseteq [n]$ and $J\not\ni i\in [n]$, whether $f_i$ and $f_J$ are coprime and whether the former divides the latter.
\end{lemma}
\noindent We claim that
\begin{lemma}\label{227}
For any monomial ideal, there is a monomial ideal sharing its cohomological support variety which can by taking a collection of subsets of $[n]$, assigning each a variable, and letting $f_i$ be the product of the variables corresponding to the subsets containing $i$. Furthermore, for any equidegree minimal generating set of a monomial ideal, there is an equidegree minimal generating set of a monomial ideal of the same size sharing its cohomological support variety which can be constructed by taking a monomial ideal as described above, assigning to each variable a positive integer, and taking each variable to the power of its corresponding integer.
\end{lemma}
\begin{proof}
Consider some monomial ideal and some variable $x_0$ of its underlying ring. Let $p_i$ be the power of $x_0$ in $f_i$. Then, recalling Lemma \ref{codenameprin}, we can replace $x_0^{p_i}$ with $\prod_{j=1}^{p_i}x_{0,j}$ for all $i$ without modifying our cohomological support variety, and we can do so without affecting minimality as this is given in the monomial setting by the non-existence of one generator dividing another. After doing this for each variable in our underlying ring, we may either remove redundant variables, those which are contained in the same subset of $\boldsymbol{\mathit{f}}$ as others, proving our first claim, or replace them with the variables they copy, proving the second.
\end{proof}

\noindent As an example, after our first step the sequence $\boldsymbol{\mathit{f}}=(x_1x_2^2,x_2^3,x_3^3)$ would become
\[(x_{1,1}x_{2,1}x_{2,2},x_{2,1}x_{2,2}x_{2,3},x_{3,1}x_{3,2}x_{3,3}),\]
at which point our second steps would modify it into either
\[(x_{1,1}x_{2,1},x_{2,1}x_{2,3},x_{3,1})\quad\text{or}\quad(x_{1,1}x_{2,1}^2,x_{2,1}^2x_{2,3},x_{3,1}^3).\]

We categorize ideals by their corresponding greatest common denominator (GCD) graphs \cite[Definition 6.1]{embdef}. As noted in \cite{embdef}, there are a number of GCD graphs such that all corresponding monomial ideals must have full support:
\begin{lemma}[{\cite[Lemma 6.12]{embdef}}]
Any monomial ideal whose GCD graph contains a vertex connected to every other vertex, or contains an edge such that every vertex is connected to exactly one of its vertices, has full support.
\end{lemma}
\noindent We use \texttt{myGraphs} to refer to the collection of graphs which is not guaranteed full support by this lemma.

In addition to these ``forbidden graphs,'' we can also determine within these graphs ``forbidden edges,'' where, if a monomial ideal constructed by Lemma \ref{codenameprin} has this GCD graph, and if there is a variable corresponding to the two vertices in that edge, then our support is full. We define these forbidden edges in terms of ``dense edges,'' such that each vertex in our graph is connected to at least one of its vertices:

\begin{lemma}
For any dense edge in our GCD graph of a monomial ideal without full support, there must be a third vertex whose monomial divides the LCM of the two comprising the dense edge in question.
\end{lemma}
\begin{proof}
If not, then the vertex of the Taylor graph corresponding to this edge is isolated, which guarantees full support by \cite[Lemma 6.9]{embdef}.
\end{proof}
\begin{corollary}\label{xecor}
For any dense edge in our GCD graph of a monomial ideal without full support, the set of variables corresponding to subsets disjoint from our dense edge must not cover all vertices connected to both vertices in our dense edge.
\end{corollary}
\begin{proof}
One of these vertices must divide the LCM of the vertices of the dense edge.
\end{proof}
\begin{lemma}\label{231}
When $n=6$ and $\boldsymbol{\mathit{f}}$ is minimal, if we have two dense edges sharing a vertex, such that the two other vertices comprising these edges are not connected in our GCD graph, then we can have no variables corresponding to edges disjoint from both of these dense edges.
\end{lemma}
\begin{proof}
Say without loss of generality that we have dense edges connecting 1 to 2 and 3, and for contradiction that our monomial ideal has an variable corresponding to the edge connecting 4 and 5. Then by Corollary \ref{xecor}, our sixth monomial must be a factor of
\[\operatorname{gcd}\{\operatorname{lcm}\{f_{\underline{1}},f_{\underline{2}}\},\operatorname{lcm}\{f_{\underline{1}},f_{\underline{3}}\}\}.\]
However, if $f_{\underline{2}}$ and $f_{\underline{3}}$ are coprime, this is simply $f_{\underline{1}}$, so $f_{\underline{6}}$ divides $f_{\underline{1}}$, contradicting minimality.
\end{proof}

We let \texttt{myCliques} refer to the set of cliques whose corresponding variables are not forbidden by Corollary \ref{xecor} or Lemma \ref{231}.

We may now describe how our test file verifies Theorem \ref{code2}. Let \texttt{myMatrix}, over $\mathbb{Q}$, take each clique in \texttt{myCliques} to its corresponding zero-one $n$-vector. The preimage of the 1-space $\left\langle\sum_{i=1}^n\mathbf{e}_i\right\rangle$ contains all formal linear sums of cliques corresponding to equigenerated monomial ideals. Let \texttt{R} (whence our useful \textit{rays} originate) be the intersection of this preimage with the space of formal linear sums with non-negative coefficients, and compute the primitive extreme rays of the resulting cone (those consisting of non-negative integers with no common factors; Macaulay2 returns these rays by default). Iteratively collect into \texttt{RAll} all sums of subsets of these rays which are not non-zero on a collection of cliques such that containing variables for each would imply full support by Corollary \ref{xecor} or Lemma \ref{231}. By Lemma \ref{227} and the fact that the cohomological support variety of a monomial ideal given by such a vector is determined by which entries are zero and which are non-zero, every non-full support of a cohomological support variety of an equigenerated monomial ideal with 6 generators must be given by one of these vectors (up to order of course, due to our choices of permutations of graphs). Use \texttt{equigeneratedMonomialCSV} to compute the cohomological support varieties of the ideals corresponding to these vectors whenever the generating sets provided are minimal. Note that, thanks to our special attention to minimality in the above lemmas, this restriction to minimal generating sets does not remove any cohomological support varieties from our result which we set out to include.

\section{Future work}\label{bgpfuture}
This work raises two natural questions.
\begin{question}\label{q2}
Can this strategy be leveraged to classify all monomial ideals with $n=6$ generators, either manually or with computer assistance? What about for other fixed $n>6$?
\end{question}
\begin{question}
Do all edge ideals of cycles in $4m+2$ variables have cohomological support variety
\[\mathcal{V}(a_1a_3\cdots a_{4m+1}+a_2a_4\cdots a_{4m+2})?\]
\end{question}

The first question is much more difficult. That said, our treatment here does not address the strategies used in \cite{embdef} to help our classification by Taylor graphs in more than identifying those with full support, and it would be interesting to see whether using these two strategies in tandem is sufficient to answer it. Furthermore, a Macaulay2 method to perform our strategy for equigenerated monomial ideals has been implemented, but one to do so for non-equigenerated monomial ideals whose Taylor subcomplex graphs admit weak gradings has not. Expansion of and experimentation with this method would be very useful in directing further inquiry. However, the second question, beyond what is known from Section \ref{inclusionsection}, progress is scant.

\printbibliography
\end{document}